\documentclass[journal]{IEEEtran}
\usepackage{amsmath, amsthm, amssymb,stfloats}
\usepackage{booktabs}
\usepackage{hyperref,setspace}
\usepackage{xcolor}

\usepackage{tikz, pgfplots}
\pgfplotsset{compat=newest}
\usetikzlibrary{arrows, positioning}
\newtheorem{definition}{\bfseries Definition}%[section]
\newtheorem{proposition}{\bfseries Proposition}%[section]
\newtheorem{theorem}{\bfseries Theorem}

\newtheorem{corollary}{\bfseries Corollary}%[section]
\newtheorem{lemma}{\bfseries Lemma}%[section]

\newtheorem{remark}{\bfseries Remark}

\renewcommand{\epsilon}{\varepsilon}
\graphicspath{{./}}

\newcommand{\ubar}{{\bf u}}
\newcommand{\ubold}{{\bf u}}
\newcommand{\vbold}{{\bf v}}
\newcommand{\xbar}{x}

\newcommand{\R}{\mathbb{R}}

\newcommand{\ep}{\epsilon}

\newcommand{\C}{\mathcal{C}}
\newcommand{\D}{\mathcal{D}}
\newcommand{\IntC}{\mathrm{Int}(\mathcal{C})}
\newcommand{\Kinfinity}{\mathcal{K}}
\newcommand{\KLinfinity}{\mathcal{KL}}
\newcommand{\Kclf}{K_{\mathrm{clf}}}
\newcommand{\Krcbf}{K_{\mathrm{rcbf}}}
\newcommand{\Kzcbf}{K_{\mathrm{zcbf}}}
\newcommand{\Fr}{F_r}
\newcommand{\gap}{\vspace{0.2cm}}
\newcommand{\newsec}[1]{\gap \noindent {\bf #1}}

\begin{document}
%	\begin{spacing}{2.0}
%\title{Control Barrier Function Based Quadratic Programs with Application to Automotive Safety Systems}
\title{Control Barrier Function Based Quadratic Programs for Safety Critical Systems}
\author{     Aaron D. Ames,
	         Xiangru Xu,
Jessy W. Grizzle,  
%	         Paulo Tabuada~\IEEEmembership{Senior Member,~IEEE}
Paulo Tabuada
	         \thanks{This research is supported by NSF CPS Awards 1239055, 1239037 and 1239085.}
%	         \thanks{A. D. Ames is with the Woodruff School of Mechanical Engineering, the
%School of Electrical and Computer Engineering,  Georgia Institute of Technology, Atlanta, Georgia,  email: aames@gatech.edu}
\thanks{A. D. Ames is with the Dept. of Mechanical and Civil Engineering, California Institute of Technology, Pasadena CA, ames@caltech.edu. X. Xu and J. W. Grizzle are with the Dept. of Electrical Engineering and Computer Science, University of Michigan, Ann Arbor, MI,  \{xuxiangr, grizzle\}@umich.edu. P. Tabuada is with the Dept. of Electrical Engineering, University of California at Los Angles, Los Angles, CA,  tabuada@ucla.edu.}
          }
%\thanks{Manuscript received April 19, 2005; revised December 27, 2012.}}

%\markboth{IEEE Transactions on Control Systems Technology,~Vol.~, No.~, Month~Year}%
%{Aaron \MakeLowercase{\textit{et al.}}: Bare Advanced Demo of IEEEtran.cls for Journals}

\maketitle

 \begin{abstract}
 Safety critical systems involve the tight coupling between potentially conflicting control objectives and safety constraints.
 As a means of creating a formal framework for controlling systems of this form, and with a view toward automotive applications,
 this paper develops a methodology that allows safety conditions---expressed as \emph{control barrier functions}---to be unified with performance objectives---expressed as control Lyapunov functions---in the context of real-time optimization-based controllers.
Safety conditions are specified in terms of forward invariance of a set, and are verified via two novel generalizations of barrier functions; in each case, the existence of a barrier function satisfying Lyapunov-like conditions implies forward invariance of the set, and the relationship between these two classes of barrier functions is characterized. In addition, each of these formulations yields a notion of control barrier function (CBF), providing inequality constraints in the control input that, when satisfied, again imply forward invariance of the set. Through these constructions, CBFs can naturally be unified with control Lyapunov functions (CLFs) in the context of a quadratic program (QP); this allows for the achievement of control objectives (represented by CLFs) subject to conditions on the admissible states of the system (represented by CBFs).
 The mediation of safety and performance through a QP is demonstrated on adaptive cruise control and lane keeping, two automotive control problems that present both safety and performance considerations coupled with actuator bounds.
  \end{abstract}

\begin{IEEEkeywords}
Control Lyapunov function, Barrier function, Nonlinear control, Quadratic program, Safety, Set invariance
%, Penalty functions
\end{IEEEkeywords}

\section{Introduction}
\label{sec:introduction}
%\input{sections/Introduction_v04.tex}

%\textbf{[Remarks for the Intro:]}
%\begin{itemize}
%\item Main focus is on CBFs
%\item Seek to combine with CLF so that performance can be achieved, when there is not a conflict with safety
%\item Mediate safety and performance through a QP
%\item Important to allow largest possible set of inputs guaranteeing safety
%\item Two types of barrier functions
%\item Prove local Lipschitz continuity of QP solution
%\item Illustrate on ACC and Lane Keeping
%\item Xu: Maybe add some literature review on barrier certificates, and the relations of those and our barriers?
%\item Xu: Important to allow largest possible set of inputs guaranteeing safety" can be incorporated into the motivation for CBFs?
%\item Xu: Mention barrier with constrained input?
%\end{itemize}
%\textbf{[End of remarks for the Intro:]}

%

Cyber-physical systems have at their core tight coupling between computation, control and physical behavior.
One of the difficulties in designing cyber-physical systems is the need to meet a large and diverse set of objectives by properly designing controllers. While it is tempting to decompose the problem into the design of a controller for each individual objective and then integrate the resulting controllers via software, the integration problem is far from being a simple one. Examples abound in, e.g., robotic and automotive systems, of unexpected and unintended interactions between controllers resulting in catastrophic behavior. In this paper we address a specific instance of this problem: how to synthesize a controller enforcing the different, and occasionally conflicting, objectives of safety and performance/stability. The overarching objective of this paper is to develop a methodology to design controllers enforcing safety objectives expressed in terms of invariance of a given set, and performance/stability objectives, expressed as the asymptotic stabilization of another given set.

%%%% Motivatino for

Motivated by the use of Lyapunov functions to certify stability properties of a set without calculating the exact solution of a system, the underlying concept in this paper is to use barrier functions to certify forward invariance of a set, while avoiding the difficult task of computing the system's reachable set. Prior work in \cite{Romd2014CDCunitiing} incorporates into a single feedback law the conditions required to simultaneously achieve asymptotic stability of an equilibrium point, while avoiding an unsafe set. Importantly, if the stabilization and safety objectives are in conflict, then no feedback law can be proposed. In contrast, the approach developed here will pose a feedback design problem that \textit{mediates} the safety and stabilization requirements, in the sense that safety is always guaranteed, and progress toward the stabilization objective is assured when the two requirements ``are not in conflict'' \cite{aaroncbfcdc14}. The essential differences in these approaches will be highlighted
through a realistic example.

%%%% Background on barrier functions and CLFs.
\subsection{\color{black}{Background}}
\label{sec:Intro:background}
\textcolor{black}{Barrier functions were first utilized in optimization; see Chapter 3 of~\cite{FGW02} for an historical account of their use in optimization. More recently, barrier functions were used in the paper~\cite{MCP16} to develop an interior penalty method for converting constrained optimal control methods into unconstrained ones\footnote{Although the techniques employed are different from ours, there are conceptual similarities as can be seen by noticing the similarity between \eqref{eqn:superlevelsetC}-\eqref{eqn:superlevelsetC3} defined later in our paper and the inequalities appearing in Proposition 4, item (g), in~\cite{MCP16} characterizing membership to the set used to define a Gauge function.}.} Barrier functions are now common throughout the control and verification literature due to their natural relationship with Lyapunov-like functions \cite{Tee,wieland2007constructive}, their ability to establish safety, avoidance, or eventuality properties \cite{aubin2009viability,Sloth2012Composite,wisniewskiconverse,Prajna2007siam,prajna2007framework}, and their relationship to multi-objective control \cite{PSV:CDC:2013}. Two notions of a barrier function associated with a set $\C$ are commonly utilized: one that is unbounded on the set boundary, i.e., $B(x) \to \infty$ as $x \to \partial \C$, termed a \textit{reciprocal barrier function} here, and one that vanishes on the set boundary, $h(x) \to 0$ as $x \to \partial \C$, called a \textit{zeroing barrier function} here. In each case, if $B$ or $h$ satisfy Lyapunov-like conditions, then forward invariance of $\C$ is guaranteed.
The natural extension of a barrier function to a system with control inputs is a Control Barrier Function (CBF), first proposed by  \cite{wieland2007constructive}. In many ways, CBFs parallel the extension of Lyapunov functions to Control Lyapunov functions (CLFs), as pioneered in \cite{Sontag:firstCLF,artstein1983stabilization,Sontag:universal} and studied in depth in \cite{FK:Book}. In each case, the key point is to impose inequality constraints on the derivative of a candidate CBF (resp., CLF) to establish entire classes of controllers that render a given set forward invariant (resp., stable).

% Importance of derivative conditions for barrier functions

%Relaxing the requirements on the derivative of a CBF enlarges the class of controllers that achieve forward invariance.

The Lyapunov-like conditions that define a (control) barrier function are intrinsically coupled to the class of controllers that achieve forward invariance of a set $\C$.
As emphasized in \cite{BarrierRevisited} and \cite{KongExpoBarrier13}, it is therefore essential to consider how one defines the evolution of a barrier function away from the set boundary, as this will translate directly to conditions imposed on a CBF.  In the case of reciprocal barrier functions, existing formulations impose invariant level sets of $B$ \cite{Tee}, via, $\dot{B} \leq 0$, as was done in earlier work on zeroing barrier functions (or barrier certificates) \cite{prajna2007framework} via $\dot{h} \geq 0$; yet, in both cases, these conditions are too restrictive on the interior of $\C$.

\subsection{ \textcolor{black}{Contributions} }
\label{sec:Intro:Contributions}

\textcolor{black}{The first contribution of this paper is to formulate conditions on the derivative of a (reciprocal or zeroing) barrier function that are minimally restrictive on the interior of $\C$. These conditions will be formulated with an eye toward their extension to control barrier functions. It is clear that less restrictive conditions for a barrier function will translate into a control barrier function that admits a larger set of inputs compatible with controlled invariance; this will be important when integrating performance with safety later in the paper. Less obvious considerations include robustness of a controlled invariant set to model perturbations, Lipschitz continuity of feedbacks achieving controlled invariance, and, as pointed out by \cite{Prajna2007siam} for barrier certificates, convexity of the set of control barrier functions when computing them numerically. }

\textcolor{black}{For reciprocal barrier functions, we allow for $B$ to grow when it is far away from the boundary of $\C$ in that we only require that $\dot{B} \leq \alpha(1/B)$, for a class-$\mathcal{K}$ function $\alpha$.  In the case of zeroing barrier functions, we adopt a condition of the form $\dot{h} \geq - \alpha(h)$. The latter condition may be somewhat surprising in view of the well-known Nagumo's Theorem, which states that for a system without inputs and a $C^1$ function $h$, the condition $\dot{h} \geq 0$ on $\partial \C$ is  necessary and sufficient for the zero superlevel set to be invariant. Importantly, under mild conditions on $\C$, it is demonstrated that the conditions we propose are also necessary and sufficient for forward invariance, and result in the relationships shown in Fig.~\ref{figbarrier2}. Moreover,  it is shown how our conditions lead to Lipschitz continuity of control laws, robustness, and convexity of the class of control barrier functions.
}

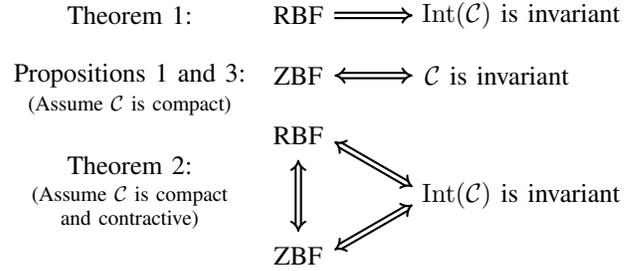
\begin{figure}[!bth]
	\centering
	\begin{tikzpicture}[scale=0.8]
	\node at (-6.5,0.5) {Theorem \ref{thm:main}:};
	\node at (-6.5,-0.5) {Propositions \ref{thm:GBF} and \ref{propnecezbf}:};
	\node at (-6.5,-1) {\footnotesize{(Assume $\C$ is compact)}};
	\node at (0,0.5) {$\mathrm{Int}(\mathcal{C})$ is invariant};
	\node at (-0.4,-0.5) {$\C$ is invariant};
	\node at (-3.7,0.5) {RBF};
	\node at (-3.7,-0.5) {ZBF};
	%    \draw [-{Straight Barb[left]}] (0,0)--(0.5,0);
	\draw[-implies,double equal sign distance,style=thick](-3.1,0.5)--(-1.8,0.5);
	\draw[implies-implies,double equal sign distance,style=thick](-3.1,-0.5)--(-1.8,-0.5);

	\node at (-6.5,-2) {Theorem \ref{thm:CBFiff}:};
	\node at (-6.5,-2.5) {\footnotesize{(Assume $\C$ is compact }};
	\node at (-6.5,-2.9) {\footnotesize{and contractive)}};
	%	\node at (-6.5,-3.8) {\footnotesize{contractive) }};
	\node at (0,-2.5) {$\mathrm{Int}(\mathcal{C})$ is invariant};
	\node at (-3.7,-1.5) {RBF};
	\node at (-3.7,-3.5) {ZBF};
	\draw[implies-implies,double equal sign distance,style=thick](-3.1,-1.6)--(-1.8,-2.35);
	\draw[implies-implies,double equal sign distance,style=thick](-3.1,-3.4)--(-1.8,-2.65);
	\draw[implies-implies,double equal sign distance,style=thick](-3.7,-1.9)--(-3.7,-3.1);
	\end{tikzpicture}
	\caption{\textcolor{black}{Relationships among reciprocal barrier functions (RBFs), zeroing barrier functions (ZBFs), and forward invariance that are developed in the paper. The underlying analysis can be found in Theorem \ref{thm:main}, Proposition \ref{thm:GBF}, Proposition \ref{propnecezbf} and Theorem \ref{thm:CBFiff}. The relations established for barrier functions then extend to control barrier functions.}
	}\label{figbarrier2}
\end{figure}

%%% Unification in QPs

Safety-critical control problems often include performance objectives, such as stabilization to a point or a surface, in addition to safety constraints.
An important novelty of the present paper is that a Quadratic Program (QP) is used to ``mediate'' these (potentially conflicting) specifications: stability and safety.
The motivation for this solution comes from \cite{AmGaGr:CDC12,AmGaGrSr:TAC12,YouTubeExp1}, which developed CLFs to exponentially stabilize periodic orbits in a class of hybrid systems. The experimental realization of CLF inspired controllers on a bipedal robot resulted in the observation that, since CLF conditions are affine in torque, they can be formulated as QPs \cite{Kevin2015torque}. Moreover, this perspective allows for the consideration of multiple control objectives (expressed via multiple CLFs) together with force- and torque-based constraints \cite{AMPO13,morris2015continuity}.  The present paper extends these ideas by unifying CBFs and CLFs through QPs.  In particular, given a control objective (expressed through a CLF) and an admissible set in the state space (expressed via a CBF), we formulate a QP that mediates the tradeoff of achieving a stabilization objective subject to ensuring the system remains in a safe set. In particular, relaxation is used to make the stability objective a soft constraint on the QP, while safety is maintained as a hard constraint. In this way, safety and stability do not need to be simultaneously satisfiable, and continuity of the resulting control law is provably maintained.

\textcolor{black}{An alternative approach to controlled invariance has been developed in \cite{mareczek2002invariance,wolff2005invariance,kimmel2014invariance,lee2012rollover,kimmel2015active}  %\cite{mareczek2002invariance,wolff2005invariance,wolff2007continuous,lee2012rollover,kimmel2014invariance,kimmel2015active}
	under the name of \textit{invariance control}. This elegant body of work is based on an extension of Nagumo's condition to functions $h$ of higher relative degree \cite{IsidoriNonControl95}, namely, it focuses on derivative conditions on the boundary of the controlled-invariant set.  As a consequence,  the control law is discontinuous, such as in sliding mode control, and as in sliding mode control, chattering may occur. We, however, establish a control framework that yields checkable conditions for Lipschitz continuous control laws and well-defined solutions of the closed-loop system. This is important from a theoretical point of view as well as the practical benefit of avoiding chattering. The consideration of the existence of solutions to the closed-loop system is one of the important differences between barrier certificates for dynamical systems and control barrier functions for control systems.}

The CBF-CLF-based QPs are illustrated on two automotive safety/convenience problems; namely,  Adaptive Cruise Control (ACC) and Lane Keeping (LK) \cite{vahidi2003research,li2011model,naus2010design,ioannou1993autonomous}.
ACC is being developed and deployed on passenger vehicles due to its promise to enhance driver convenience, safety, traffic flow, and fuel economy \cite{liang1999optimal,liang2000string,van2006impact}.  It is a multifaceted control problem because it involves asymptotic performance objectives (drive at a desired speed), subject to safety constraints (maintain a safe distance from the car in front of you), and constraints based on the physical characteristics of the car and road surface (bounded acceleration and deceleration). A key challenge is that the various objectives can often be in conflict, such as when the desired cruising speed is faster than the speed of the leading car, while provably satisfying the safety-oriented constraints is of paramount importance. Lane keeping, maintaining a vehicle between the lane markers \cite{GerdLanePotential06}, is another safety-related problem that we use to illustrate the methods developed in this paper.

A preliminary version of this work was presented in the conference publications \cite{aaroncbfcdc14} and \cite{Xu2015ADHS}. The present paper adds to those two papers in the following important ways: the relations between the two forms of barrier functions are characterized;
%Lipschitz continuity of the QP-CBF-CLF formulation is established for zeroing CBFs;
barriers with a higher relative degree are considered; the adaptive cruise control problem is extended from the lead vehicle's speed being constant to the more realistic case of varying speed with bounded input force; and the lane keeping problem is considered under the proposed QP framework.

\subsection{ \textcolor{black}{Organization and Notation} }
\label{sec:Intro:OrganizationNotatiion}
The remainder of the paper is organized as follows. Two barrier functions, specifically, reciprocal barrier functions and zeroing barrier functions, are formulated in Sect. \ref{sec:controlbarrier}, and are extended to control barrier functions in Sect. \ref{sec:controlbarrierExtensions}.  Quadratic programs that unify control Lyapunov functions and control barrier functions are introduced in section \ref{sec:lyapunovfunction}. The theory developed in the paper is illustrated on the adaptive cruise control and lane keeping problems in Sect. \ref{sec:accviaqp}, with simulations reported in Sect. \ref{sec:simulation}. Conclusions are provided in Sect. \ref{sec:conclusions}.

\vspace{.3cm}

%\textbf{[To be added:]}
\emph{Notation:} $\R,\R_0^+$ denote the set of real, non-negative real numbers, respectively. $\mathrm{Int}(\C)$ and $\partial\C $ denote the interior and boundary of the set $\C$, respectively. The open ball in $\R^n$ with radius $\varepsilon\in \mathbb{R}^+$ and center at $0$ is denoted by $B_\varepsilon=\{x\in \mathbb{R}^n\,\vert\, \Vert x\Vert< \varepsilon\}$. The Minkowsky sum of two sets $\mathcal{R} \subseteq \mathbb{R}^n$ and $\mathcal{S} \subseteq \mathbb{R}^n$ is denoted by $\mathcal{R}\oplus\mathcal{S}$. The distance from $x$ to a set $\mathcal{S}$ is denoted by \mbox{$\Vert x\Vert_\mathcal{S}=\inf_{s\in \mathcal{S}}\Vert x-s\Vert$}.
For any essentially bounded function $g:\mathbb{R}\to \mathbb{R}^n$,  the infinity norm of $g$ is denoted by $\Vert g\Vert_{\infty}=\textrm{ess}\sup_{t\in \mathbb{R}}\Vert g(t)\Vert$. A continuous function \mbox{$\beta_1:[0,a)\rightarrow [0,\infty)$} for some $a>0$ is said to belong to \emph{class $\mathcal{K}$} if it is strictly increasing and $\beta_1(0)=0$.
A continuous function \mbox{$\beta_2:[0,b)\times [0,\infty)\rightarrow [0,\infty)$} for some $b>0$ is said to belong to \emph{class $\mathcal{KL}$}, if for each fixed $s$, the mapping $\beta_2(r,s)$ belongs to class $\mathcal{K}$ with respect to $r$ and for each fixed $r$, the mapping $\beta_2(r,s)$ is decreasing with  respect to $s$ and $\beta_2(r,s)\rightarrow 0$ as $s\rightarrow \infty$.

\section{Reciprocal and Zeroing Barrier Functions}
\label{sec:controlbarrier}
This section studies two notions of barrier functions and investigates their relationships with forward invariance of a set. Consider a nonlinear system of the form
\begin{eqnarray}
	\label{eqn:dynamicalsystem}
	\dot{x} = f(x)
\end{eqnarray}
where $x \in \R^n$ and $f$ is assumed to be locally Lipschitz.   Then for any initial condition $x_0:=x(t_0) \in \R^n$, there exists a maximum time interval $I(x_0) = [t_0,\tau_{\mathrm{max}})$ such that $x(t)$ is the unique solution to \eqref{eqn:dynamicalsystem} on $I(x_0)$; in the case when $f$ is forward complete, $\tau_{\mathrm{max}} = \infty$.  A set $\mathcal{S}$ is called {\it (forward) invariant} with respect to \eqref{eqn:dynamicalsystem} if for every $x_0 \in \mathcal{S}$, $x(t) \in \mathcal{S}$ for all $t \in I(x_0)$.
% {\color{blue}In general, vector fields are not complete and hence $ I(x_0)\neq [t_0, \infty)$.}
%Given a closed set $\C \subset \R^n$, we determine conditions on functions $B : \C \to \R$ such that solutions of \eqref{eqn:dynamicalsystem} are guaranteed to remain within $\C$, that is, set $\C$ is forward invariant.
%These conditions will motivate the formulation of control barrier functions.

%SWe begin by first motivating the constructions through a concrete example of a barrier function.

\subsection{Reciprocal Barrier Functions}
\subsubsection{Motivation}
\label{sec:motivation}
%Let's assume that solutions $x(t)$ of \eqref{eqn:dynamicalsystem} are forward complete and suppose that we have a closed set $\C$ for which we wish to verify that $x(t) \in \mathrm{Int}({\C})$ for all $t \geq 0$.  For simplicity, suppose furthermore that
%\begin{eqnarray}
%\label{eqn:superlevelsetC}
%\C &=& \{ x \in \R^n : h(x) \geq 0\}, \\
%\label{eqn:superlevelsetC2}
%\partial \C &=& \{ x \in \R^n : h(x) = 0\}, \\
%\label{eqn:superlevelsetC3}
%%\mathring{\mathcal{C}}&=& \{ x \in \R^n : h(x) > 0\},
%\mathrm{Int}(\C) &=& \{ x \in \R^n : h(x) > 0\},
%\end{eqnarray}

Given a closed set $\C \subset \R^n$, we determine conditions on functions $B : \mathrm{Int}({\C}) \to \R$ such that $\mathrm{Int}({\C})$ is forward invariant.
%solutions of \eqref{eqn:dynamicalsystem} are guaranteed to remain within $\C$, that is,  $x(t) \in \mathrm{Int}({\C})$ for all $t \geq 0$.
These conditions will motivate the formulation of the barrier functions considered in this paper.

%{\color{red} [Comments by Xiangru: Should we impose some regularity condition on $h$ such that $\C$ has some ``nice'' property (although our proof has no requirement on $h$ other than the continuously differentiability)? But if the boundary of $\C$ intersect with itself, there would be some singular point on $\partial\C$, which means that $\nabla h=\textbf{0}$ at that point; if there exists some isolated point, $\mathrm{Int}({\C})$ would be empty.
%%On the other hand, Nagumo's theorem states that the set $\C$ is forward invariant if and only if $f\in\mathcal{T}_{\C}(x),\forall x\in\partial\C$ where $\mathcal{T}_{\C}(x)$ is the tangent cone of $\C$. If the boundary of $\C$ at some $x\in\partial\C$ is smooth, $\mathcal{T}_{\C}(x)$ is just the tangent halfplane and condition $f\in\mathcal{T}_{\C}(x)$ can be stated as $\dot{h}=\nabla h^Tf\geq 0$; but at the singular point, $\mathcal{T}_{\C}(x)$ is no longer the tangent halfplane. Can we still use $\nabla h(x)^Tf\geq 0$ as the necessary condition for the invariance of $\C$ (in the proof of Lemma 3)?]
%]}

Assume that the set $C$ is defined as
\begin{align}
	\label{eqn:superlevelsetC}
	\C &= \{ x \in \R^n : h(x) \geq 0\}, \\
	\label{eqn:superlevelsetC2}
	\partial \C &= \{ x \in \R^n : h(x) = 0\}, \\
	\label{eqn:superlevelsetC3}
	%\mathring{\mathcal{C}}&=& \{ x \in \R^n : h(x) > 0\},
	\mathrm{Int}(\C) &= \{ x \in \R^n : h(x) > 0\},
\end{align}
where \mbox{$h: \R^n \to \R$} is a continuously differentiable function. Later, it will also be assumed that $\C$ is nonempty and has no isolated point, that is,
\begin{equation}
	\label{eqn:closurePropertyC}
	\mathrm{Int}({\C}) \not= \emptyset ~\text{and}~~\overline{\mathrm{Int}(\mathcal{C})} = \mathcal{C}.
\end{equation}

Motivated by the barrier method in optimization \cite{boyd2004convex}, consider the logarithmic barrier function candidate
\begin{eqnarray}
	\label{eqn:Bmotivation}
	B(x) = - \log\left( \frac{h(x)}{1 + h(x)} \right).
\end{eqnarray}
Note that this function satisfies the important properties
\begin{eqnarray}
	\inf_{x \in \mathrm{Int}(\mathcal{C})}  B(x)  \geq  0 , \qquad
	\lim_{x \to \partial\mathcal{C}}  B(x)   =  \infty.   \label{barrierproperty}
\end{eqnarray}

The question then becomes: {\it what conditions should be imposed on $\dot{B}$ so that $\mathrm{Int}(\mathcal{C})$ is forward invariant? } The conventional answer in \cite{Tee,prajna2007framework} has been to enforce the condition $\dot{B} \leq 0$, but this may not be desirable since it requires all sublevel sets of $\C$ to be invariant;  in particular, it will not allow a solution to leave a sublevel set even if by doing so it will remain in $\IntC$.  A condition analogous to this was relaxed by \cite{KongExpoBarrier13} and \cite{BarrierRevisited} where the key idea was to only require a single sublevel set to be invariant. Motivated by this, we relax the condition $\dot{B} \leq 0$ to
\begin{eqnarray}
	\label{eqn:barriercond}
	\dot{B} \leq \frac{\gamma}{B},
\end{eqnarray}
where $\gamma$ is positive. This inequality allows for $\dot{B}$ to grow when solutions are far from the boundary of $\C$.  As solutions approach the boundary, the rate of growth decreases to zero.

For \eqref{eqn:barriercond} to be an acceptable condition, we need to verify that its satisfaction guarantees that solutions to \eqref{eqn:dynamicalsystem} stay in $\mathrm{Int}(\mathcal{C})$.  To see this, we note that differentiating \eqref{eqn:Bmotivation} along solutions of \eqref{eqn:dynamicalsystem} gives
$$
\dot{B} = - \frac{\dot{h}}{h + h^2}.
$$
Therefore, \eqref{eqn:barriercond} implies that the rate of change in $h$ with respect to $t$ is bounded by
$$
\dot{h} \geq  \frac{ \gamma ( h + h^2)}{\log\left(\frac{h}{1 + h}\right)}.
$$
Assuming for the moment that solutions $x(t,x_0)$ of \eqref{eqn:dynamicalsystem} are forward complete, the Comparison Lemma \cite{KHALIL01} implies that
%: \textbf{[Remark by Xiangru for AA: Please check the power of 2 on the $\log$ function.]}
$$
h(x(t,x_0)) \geq \frac{1}{-1 + \exp\left({\sqrt{ 2 \gamma t + \log^2\left(\frac{h(x_0) + 1}{h(x_0)}\right) }}\right)}.
$$
Therefore, if $h(x_0) > 0$, i.e., $x_0 \in \IntC$, then condition \eqref{eqn:barriercond} guarantees that \mbox{$h(x(t,x_0)) > 0$} for all $t \geq 0$, i.e., \mbox{$x(t,x_0) \in \mathrm{Int}(\C)$} for all $t \geq 0$.

Apart from \eqref{eqn:Bmotivation}, another barrier function that is commonly considered in optimization is the inverse-type barrier candidate
\begin{eqnarray}
	\label{eqn:barrierfunc2}
	B(x) =  \frac{1}{h(x)}.
\end{eqnarray}
Note that $B(x)$ in \eqref{eqn:barrierfunc2} also satisfies the properties in \eqref{barrierproperty}.
%\mbox{$\inf_{x \in \mathrm{Int}(\mathcal{C})}B(x)\geq 0$},$\lim_{x \to \partial\mathcal{C}}B(x)=\infty$.
If condition \eqref{eqn:barriercond} holds, then by the Comparison Lemma, we have
$$
h(x(t,x_0))\geq \frac{1}{\sqrt{2\gamma t+\frac{1}{h^2(x_0)}}},
$$
and once again, \mbox{$x(t,x_0) \in \mathrm{Int}(\C)$} for all $t \geq 0$, provided that $x_0 \in \mathrm{Int}(C)$.
%$h(x(t)) > 0$ for all $t \geq 0$ provided \mbox{$h(x(0)) > 0$}.

\subsubsection{Reciprocal Barrier Functions and Set Invariance}

Based on the presented motivation, we formulate a notion of barrier function that provides the same guarantees in a more general context.
%The definition of class $\mathcal{K}$ and class $\mathcal{KL}$ function can be found in \cite{KHALIL01}.

%\textbf{[JWG to Xiangru: I added B differentiable. Do we need h differentiable in the following definition?  See also Thm. 1]}

\begin{definition}
	\label{def:barrierfunctions}
	%Given the dynamical system \eqref{eqn:dynamicalsystem} and the set $\C$ defined by \eqref{eqn:superlevelsetC}-\eqref{eqn:superlevelsetC3}
	%for a continuously differentiable function $h: \R^n \to \R$, if
	For the dynamical system \eqref{eqn:dynamicalsystem}, a continuously differentiable function \mbox{$B : \mathrm{Int}(\mathcal{C})  \to \R$} is a \emph{reciprocal barrier function (RBF)} for the set $\C$ defined by \eqref{eqn:superlevelsetC}-\eqref{eqn:superlevelsetC3} for a continuously differentiable function $h: \R^n \to \R$, if there exist class $\Kinfinity$ functions $\alpha_1,\alpha_2,\alpha_3$ such that, for all $x \in \mathrm{Int}(\mathcal{C}) $,
	\begin{align}
		\label{eqn:Binequality}
		\frac{1}{\alpha_1(h(x))}   \leq   B(x)  &  \leq \frac{1}{\alpha_2(h(x) ) },\\
		\label{eqn:Bdotinequality}
		L_f{B}(x) &   \leq \alpha_3(h(x)).
	\end{align}
	
\end{definition}

\begin{remark}
	The Lyapunov-like bounds \eqref{eqn:Binequality} on $B$ imply that along solutions of \eqref{eqn:dynamicalsystem}, $B$ essentially behaves like $\frac{1}{\alpha(h)}$ for some class $\Kinfinity$ function $\alpha$ with
	\begin{eqnarray}
		\inf_{x \in \mathrm{Int}(\mathcal{C})} \frac{1}{\alpha(h(x) )}   \geq  0 , \quad
		\lim_{x \to \partial\mathcal{C}} \frac{1}{\alpha(h(x))}    =  \infty.   \nonumber
	\end{eqnarray}
	The condition \eqref{eqn:Bdotinequality} on $\dot{B}=L_fB$, which generalizes condition \eqref{eqn:barriercond}, allows for $B$ to grow quickly when solutions are far away from $\partial \C$, with the growth rate approaching zero as solutions approach $\partial \C$.
\end{remark}

\begin{remark} In the conference version \cite{aaroncbfcdc14}, a function satisfying Def. \ref{def:barrierfunctions} was  simply called a \emph{barrier function} and not a \emph{reciprocal barrier function}. The new terminology is necessary to make the distinction with a second type of barrier function used in the next subsection.
\end{remark}

%\newsec{Main result.}
%The notion of barrier functions introduced allows us to state the main result of this paper.  First, we

\begin{theorem}
	\label{thm:main}
	%{\it
	Given a set $\C \subset \R^n$ defined by \eqref{eqn:superlevelsetC}-\eqref{eqn:superlevelsetC3} for a continuously differentiable function $h$, if there exists a RBF $B : \mathrm{Int}(\mathcal{C}) \to \R$, then $\mathrm{Int}(\mathcal{C})$ is forward invariant.
\end{theorem}

The following lemma is established to prove Theorem \ref{thm:main}.

%\gap

\begin{lemma}\label{lem:simplesystem}
	Consider the dynamical system
	\begin{eqnarray}
		\label{eqn:dotyeqn}
		\dot{y} = \alpha\left( \frac{1}{y} \right), \qquad y(t_0) = y_0,
	\end{eqnarray}
	with $\alpha$ a class $\Kinfinity$ function. For every  \mbox{$y_0 \in (0,\infty)$}, the system has a unique solution defined for all $t \geq t_0$ and given by
	\begin{eqnarray}
		\label{eqn:ytsolutionform}
		y(t) = \frac{1}{\sigma\left( \frac{1}{y_0} , t - t_0 \right)},
	\end{eqnarray}
	where $\sigma$ is a class $\KLinfinity$ function.
\end{lemma}

%\gap

\begin{proof}
	Under the change of variables $z = \frac{1}{y}$, the dynamical system \eqref{eqn:dotyeqn} becomes
	\begin{align}\label{eqn:lemma1}
		\dot{z} = -\frac{\dot{y}}{y^2} = - \frac{\alpha\left( \frac{1}{y} \right)}{y^2} = -\alpha(z) z^2 := - \bar{\alpha}(z).
	\end{align}
	Since $\alpha(z)$ is a class $\Kinfinity$ function, it follows that $\bar{\alpha}(z) = \alpha(z) z^2$ is a class $\Kinfinity$ function. The fact that $\bar{\alpha}(z)$ is a continuous, non-increasing function for all $z\ge 0$ implies that \eqref{eqn:lemma1} has a unique solution for every initial state $z_0>0$; see Peano's Uniqueness Theorem (Thm. 1.3.1 in \cite{agarwal1993uniqueness}, Thm. 6.2  in \cite{hartman2002ODE}). Furthermore, by the proof of Lemma 4.4 of \cite{KHALIL01}, it follows that the solution is defined on $[t_0, \infty)$ and is given by
	$$
	z(t) = \sigma(z_0, t-t_0),
	$$
	with $\sigma$ a class $\KLinfinity$ function.  Converting from $z$ back to $y$ through $y = \frac{1}{z}$ yields the solution $y(t)$ given in \eqref{eqn:ytsolutionform}.
\end{proof}

%\gap

We now have the necessary framework in which to prove Theorem \ref{thm:main}.

%\gap

\begin{proof} (of Theorem \ref{thm:main})
	Utilizing \eqref{eqn:Binequality} and \eqref{eqn:Bdotinequality}, we have that
	\begin{equation}\label{eqn:BdotIneq}
		\dot{B} \leq \alpha_3 \circ \alpha_2^{-1} \left( \frac{1}{B} \right) := \alpha\left( \frac{1}{B} \right).
	\end{equation}
	Since the inverse of a class $\Kinfinity$ function is a class $\Kinfinity$ function, and the composition of class $\Kinfinity$ functions is a class $\Kinfinity$ function, $\alpha = \alpha_3 \circ \alpha_2^{-1}$ is a class $\Kinfinity$ function.
	
	Let $x(t)$ be a solution of \eqref{eqn:dynamicalsystem} with $x_0 \in \mathrm{Int}(\mathcal{C})$, and let $B(t) = B(x(t))$. The next step is to apply the Comparison Lemma to \eqref{eqn:BdotIneq} so that  $B(t)$ is upper bounded by the solution of \eqref{eqn:dotyeqn}. To do so, it must be noted that the hypothesis ``$f(t,u)$ is locally Lipschitz in $u$'' used in the proof of Lemma 3.4 in \cite{KHALIL01},  can be replaced by with the hypothesis ``$f(t,u)$ is continuous, non-increasing in $u$''. This is valid because the proof only uses the local Lipschitz assumption to obtain uniqueness of solutions to \eqref{eqn:dotyeqn}, and this was taken care of with  Peano's Uniqueness Theorem in the proof of Lemma \ref{lem:simplesystem}.
	
	Hence, the Comparison Lemma in combination with Lemma \ref{lem:simplesystem} yields
	\begin{equation}\label{eqn:barrierineq}
		B(x(t)) \leq \frac{1}{\sigma\left( \frac{1}{B(x_0)} , t - t_0 \right)},
	\end{equation}
	for all $t \in I(x_0)$, where $x_0=x(t_0)$.  This, coupled with the left inequality in \eqref{eqn:Binequality}, implies that
	\begin{eqnarray}
		\label{eqn:maininequality}
		\alpha^{-1}_1\left(\sigma\left( \frac{1}{B(x_0)} , t - t_0 \right)\right)  \leq  h(x(t)),
	\end{eqnarray}
	for all $t \in I(x_0)$.  By the properties of class $\Kinfinity$ and $\KLinfinity$ functions, if $x_0 \in \mathrm{Int}(\mathcal{C})$ and hence $B(x_0) > 0$, it follows from \eqref{eqn:maininequality} that $ h(x(t)) > 0$ for all $t \in I(x_0)$.  Therefore, $x(t) \in \mathrm{Int}(\mathcal{C})$ for all $t \in I(x_0)$, which implies that $\mathrm{Int}(\mathcal{C})$ is forward invariant.
	%\begin{align}
	%   \sigma\left( \frac{1}{B(x(t_0))} , t - t_0 \right) & \leq  \alpha_1( \| x\|_{\partial \C} ) \nonumber\\
	%\Rightarrow  \qquad  \alpha^{-1}\left(\sigma\left( \frac{1}{B(x(t_0))} , t - t_0 \right)\right) & \leq  \| x\|_{\partial \C}  \nonumber
	%\end{align}
\end{proof}

\begin{remark}\label{remBarrier}
	Inequality \eqref{eqn:barrierineq} is the essential condition to make $B$ a reciprocal barrier function, because it ensures that $B(x(T))\neq\infty$ for any finite $T \in I(x_0)$, which implies that $h(x(T))> 0$ for any $T\in I(x_0)$ if $h(x_0)>0$.
	%This is equivalent to $h(x(T))\neq 0$ for any finite $T \in I(x_0)$, which guarantees that $h(x(T))> 0$ for any $T\in I(x_0)$ provided that $h(x_0)>0$.
\end{remark}

\begin{remark}\label{remBF}
	Note that the function considered in \eqref{eqn:Bmotivation}, subject to the condition \eqref{eqn:barriercond} for some $\gamma>0$, is a RBF by Def. \ref{def:barrierfunctions}. %and therefore Theorem \ref{thm:main} establishes the invariance of the set $\C$.
	This follows from the fact that
	%$$
	%\alpha(r) = \frac{1}{- \log\left( \frac{r}{1 + r} \right)}
	%$$
	$$
	\alpha(r) =\begin{cases} \frac{1}{- \log\left( \frac{r}{1 + r} \right)} &\mbox{if }r >0 \\
	0 & \mbox{if } r=0. \end{cases}
	$$
	is a class $\Kinfinity$ function.  Therefore, in Def. \ref{def:barrierfunctions}, we choose $\alpha_1(r) = \alpha_2(r) = \alpha(r)$ and $\alpha_3(r) = \gamma \alpha(r)$. Note also that the function \eqref{eqn:barrierfunc2} satisfying \eqref{eqn:barriercond} is also a RBF with class $\Kinfinity$ functions $\alpha_1(r) = \alpha_2(r) = r$ and $\alpha_3(r) = \gamma r$.
\end{remark}

\subsection{Zeroing Barrier Functions}
Intrinsic to the notion of RBF is the fact, formalized in \eqref{eqn:Binequality}, that such a function tends to plus infinity as its argument approaches the boundary of $C$. Unbounded function values, however, may be undesirable when real-time/embedded implementations are considered. Motivated by this and the barrier certificates in \cite{KongExpoBarrier13}, we study a barrier function that vanishes on the boundary of the set $\C$. This is facilitated by fist defining the notion of an extended class $\mathcal{K}$ function.

%For $\epsilon\geq 0$, define the family of closed sets $\Ce$ as
%\begin{eqnarray}
%\label{eqn:newsuperlevelsetC}
%\Ce &=& \{ x \in \R^n : h(x) \geq -\ep\}, \\
%\label{eqn:newsuperlevelsetC2}
%\partial \Ce &=& \{ x \in \R^n : h(x) = -\ep\}, \\
%\label{eqn:newsuperlevelsetC3}
%\mathrm{Int}(\Ce) &=& \{ x \in \R^n : h(x) > -\ep\},
%\end{eqnarray}
%where \mbox{$h: \R^n \to \R$} is a continuously differentiable function. By construction, $\C_{\ep_1}\subset\C_{\ep_2}$ for any $\ep_2>\ep_1\geq 0$. For simplicity, the set $\C_0$ is denoted by $\C$.

\begin{definition}\label{def:extend}
	A continuous function \mbox{$\alpha:(-b,a)\rightarrow (-\infty,\infty)$} is said to belong to \emph{extended class $\mathcal{K}$} for some $a,b>0$ if it is strictly increasing and $\alpha(0)=0$. \\
\end{definition}

\begin{definition}
	\label{def:barrierfunctions2}
	For the dynamical system \eqref{eqn:dynamicalsystem}, a continuously differentiable function $h: \R^n \to \R$ is a \emph{zeroing barrier function (ZBF)} for the set $\C$ defined by \eqref{eqn:superlevelsetC}-\eqref{eqn:superlevelsetC3}, if there exist an extended class $\Kinfinity$ function $\alpha$ and a set $\D$ with $\C\subseteq\mathcal{D}\subset \R^n$ such that, for all $x \in \D$,
	%for a continuously differentiable function $h: \R^n \to \R$,a continuously differentiable function \mbox{$B : \mathrm{Int}(\mathcal{C})  \to \R$} is a {\emph reciprocal barrier function (RBF)} for the set $\C$ defined by \eqref{eqn:superlevelsetC}-\eqref{eqn:superlevelsetC3} for a continuously differentiable function $h: \R^n \to \R$,
	%Given the dynamical system \eqref{eqn:dynamicalsystem} and the set $\C$ defined by \eqref{eqn:superlevelsetC}-\eqref{eqn:superlevelsetC3}
	%for a continuously differentiable function $h: \R^n \to \R$, if there exist an extended class $\Kinfinity$ function $\alpha$ and a set $\D$ with $\C\subseteq\mathcal{D}\subset \R^n$ such that
	\begin{align}
		L_f{h}(x)& \geq  -\alpha(h(x)). \label{eqn:ZBFinequality}
	\end{align}
	%then the function $h$ is called a \emph{zeroing barrier function (ZBF)}.
	%
	%a continuously differentiable function $\phi: \mathcal{D}\rightarrow \R$ is a {\bf zeroing barrier function (ZBF)} defined on set $\mathcal{D}$ where $\C\subset\mathcal{D}\subset \R^n$, if there exists a locally Lipschitz extended class $\Kinfinity$ function $\alpha_0$
	%such that
	%defined on $(\inf_{x\in\mathcal{D}}\phi(x),sup_{x\in\mathcal{D}}\phi(x))$ such that
	%\begin{align}
	%\phi(x)&>0,\;\forall x\in \mathrm{Int}(\C)\label{eqn:generalinequality1}\\
	%\phi(x)&=0,\;\forall x\in \partial\C\label{eqn:generalinequality2}\\
	%\phi(x)&<0,\;\forall x\in \mathcal{D}\backslash\C\label{eqn:generalinequality4}\\
	%\dot{\phi}(x)& \geq  -\alpha_0(\phi(x)),\forall \; x \in \mathcal{D}.\label{eqn:generalinequality3}
	%\end{align}
\end{definition}

\begin{remark}
	Defining $h$ on a set $\mathcal{D}$ larger than $\C$ allows one to consider the effects of model perturbations. This idea is developed in the conference submission \cite{Xu2015ADHS}, where it is also illustrated on a realistic problem.
\end{remark}

\textcolor{black}{
	\begin{remark}
		A special case of \eqref{eqn:ZBFinequality} is
		\begin{align}
			L_f{h}(x)& \geq  -\gamma h(x), \label{eqn:ZBFinequalityLinear}
		\end{align}
		for $\gamma>0$. This leads to a convex problem when seeking barrier functions with numerical means, such as sum of squares (SOS) \cite{Prajna2007siam}.
	\end{remark}
}

Similar to Theorem \ref{thm:main}, existence of a ZBF implies the forward invariance of $\mathcal{C}$, as shown by the following theorem.
%\textbf{[This is basically known? I think it should be called a proposition?]}

\begin{proposition}\label{thm:GBF}
	Given a dynamical system \eqref{eqn:dynamicalsystem} and a set $\C$ defined by \eqref{eqn:superlevelsetC}-\eqref{eqn:superlevelsetC3} for some continuously differentiable function $h: \R^n\rightarrow \R$, if $h$ is a ZBF defined on the set $\D$ with $\C\subseteq\mathcal{D}\subset \R^n$, then $\C$ is forward invariant.
\end{proposition}
\begin{proof}
	Note that for any $x \in \partial\C$, $\dot{h}(x)\geq -\alpha(h(x))=0$. According to Nagumo's theorem \cite{BlanchiniBook08,BlanchiniAuto99}, the set $\C$ is forward invariant.
	%Similarly, for any $\ep\geq 0$ such that $\C_\ep\subseteq\mathcal{D}$,  for any $x \in \partial\Ce$, we have $L_f h(x)\geq -\alpha(h(x))=-\alpha(-\ep)>0$. Again, Nagumo's theorem implies that the set $\Ce$ is forward invariant.
\end{proof}

\begin{remark}
	As stated in  Remark \ref{remBarrier}, what makes function $B$ of Def. \ref{def:barrierfunctions} a barrier is that $B(x(T)) < \infty$ for any finite $T\in I(x(t_0))$. Here, what makes function $h$ of Def. \ref{def:barrierfunctions2} a barrier is that $h(x(T))>0$ for any finite $T\in I(x(t_0))$.
\end{remark}

%\subsubsection{Robustness Properties of Zeroing Barrier Functions}
%
%In this subsection,
%Next, we will show the extent to which forward invariance of the set $\C$, asserted in Theorem~\ref{thm:GBF}, is robust with respect to different perturbations on the dynamics~\eqref{eqn:dynamicalsystem} is investigated. This will be accomplished by showing that existence of a ZBF implies asymptotic stability of the set $\C$.

%Recall that a closed and forward invariant set $\mathcal{S}\subseteq \mathbb{R}^n$ is  locally asymptotically stable for a forward complete system~\eqref{eqn:dynamicalsystem}, if there exist an open set $\mathcal{R}$ containing $\mathcal{S}$ and a class $\mathcal{KL}$ function $\beta$ such that for any $x_0\in \mathcal{R}$,
%\begin{equation*}
%\label{Eq:SetStability}
%\Vert x(t,x_0)\Vert_{\mathcal{S}}\le \beta\left(\Vert x_0\Vert_{\mathcal{S}},t\right).
%\end{equation*}

%Whenever the set $\mathcal{S}$ is compact, inequality~\eqref{Eq:SetStability} implies $I(x_0)=\mathbb{R}_0^+$ for all $x_0\in \mathcal{R}$. Therefore, the forward completeness assumption on~\eqref{eqn:dynamicalsystem} is no longer needed.

For a ZBF $h$ defined on a set $\D$, if $\mathcal{D}$ is open, then $h$ induces a Lyapunov function $V_\C:\mathcal{D}\to \mathbb{R}_0^+$ defined by
\begin{equation}
	\label{Eq:V}
	V_\C(x)=\left\{\begin{array}{ccl}
		0, & \text{if} & x\in \C,\\
		-h(x), & \text{if} & x\in\mathcal{D}\backslash\C.\end{array}\right.
\end{equation}
It is easy to see that: \textbf{1)} $V_\C(x)=0$ for $x\in \C$; \textbf{2)} $V_\C(x)>0$ for $x\in\mathcal{D}\backslash\C$; and \textbf{3)} $L_f V_\C(x)$ satisfies the following inequality for $x\in\mathcal{D}\backslash \C$:
$$L_f V_\C(x) =  -L_f h(x)\le\alpha\circ h(x) = \alpha(-V_\C(x))<0,$$
where $\alpha$ is the extended class $\Kinfinity$ function introduced in Def. \ref{def:barrierfunctions2}.
%\textbf{[Xiangru: Is $V$ a ``good'' Lyapunov function?  the class $K$ function of $\|x\|$ that upper and lower bound $V$ can not be easily found (for computation purposes)?]}
It thus follows from these three properties, from the fact that $V_\C$ is continuous on its domain and continuously differentiable at every point $x\in\mathcal{D}\backslash \C$, and from\footnote{While Theorem 2.8 requires the function $V$ to be smooth, $V$ can always be smoothed as shown in Proposition 4.2 in~\cite{SmoothConverse96}.} Theorem 2.8 in~\cite{SmoothConverse96} that the set $\C$ is asymptotically stable whenever~\eqref{eqn:dynamicalsystem} is forward complete or the set $\C$ is compact. This is summarized in the following result.

\begin{proposition}
	\label{Prop:SetStability}
	Let $h:\mathcal{D}\to\mathbb{R}$ be a continuously differentiable function defined on an open set $\mathcal{D}\subseteq\mathbb{R}^n$. If $h$ is a ZBF for the dynamical system~\eqref{eqn:dynamicalsystem}, then the set $\C$ defined by $h$ is asymptotically stable. Moreover, the function $V_\C$ defined in \eqref{Eq:V} is a Lyapunov function.
	%Lyapunov function~\eqref{Eq:V} implies asymptotic stability of the set  $\C$  whenever~\eqref{eqn:dynamicalsystem} is forward complete or the set $\C$ is compact.
\end{proposition}

Note that asymptotic stability of $\mathcal{C}$ implies forward invariance of $\mathcal{C}$ as described in \cite{Xu2015ADHS}.
%: $x_0\in\mathcal{C}$ implies $\Vert {\color{black}x_0}\Vert_\mathcal{C}=0$ and $\beta(\Vert {\color{black}x_0}\Vert_\mathcal{C},t)=0$ which, in turn, implies $\Vert x(t,x_0)\Vert_\mathcal{C}{\color{black}=} 0$ and $x(t,x_0)\in \mathcal{C}$.
%Therefore, by showing that existence of a ZBF implies asymptotic stability of the set $\C$, the extent to which forward invariance of the set $\C$  is robust with respect to different perturbations on the dynamics~\eqref{eqn:dynamicalsystem} can be investigated. Indeed,
%Once the asymptotic stability of $\C$ is established by Proposition \ref{Prop:SetStability},
Therefore, existing robustness results in the literature (such as \cite{Bacc2005Lyapunovfunctionbook,Isidori1999NonlinearBook2}) can be used to characterize  the extent to which forward invariance of the set $\C$  is robust with respect to different perturbations on the dynamics~\eqref{eqn:dynamicalsystem}. The reader is referred to \cite{Xu2015ADHS} for further discussion and an application.
%Based on Proposition \ref{Prop:SetStability}, asymptotic stability and robustness properties can be explored using existing results form the literature.

\subsection{Relationships of RBFs, ZBFs and Set Invariance}\label{subsec:necesuff}
Theorem \ref{thm:main} and Prop. \ref{thm:GBF} in the two previous subsections show that the existence of a RBF (resp., a ZBF) is a sufficient condition for the forward invariance of $\mathrm{Int}(\mathcal{C})$ (resp., $\C$). This section investigates cases where the converse holds and other relations among these two types of barrier functions.

\begin{proposition}\label{propnecezbf}
	Consider the dynamical system \eqref{eqn:dynamicalsystem} and a nonempty, compact set $\C$ defined by \eqref{eqn:superlevelsetC}-\eqref{eqn:superlevelsetC3}
	for a continuously differentiable function $h$. If $\C$  is forward invariant,  then $h\vert_\C$ is a ZBF defined on $\C$.
\end{proposition}
\begin{proof}
	We take $\mathcal{D}=\C$ in Def. \ref{def:barrierfunctions2}.
	For any $r \geq0$, the set $\{x|0\leq h(x)\leq r\}$ is a compact subset of $\C$. Define a function $\alpha:[0,\infty)\rightarrow \R$ by
	\begin{align*}
		\alpha(r)=-\inf_{\{x|0\leq h(x)\leq r\}}L_fh(x).
	\end{align*}
	Using the compactness property stated above and the continuity of $L_fh$, $\alpha$ is a well defined, non-decreasing function on $\R_0^+$ satisfying
	$$
	L_fh(x)\geq -\alpha\circ h(x),\;\forall x\in\mathcal{C}.
	$$
	
	By Nagumo's theorem \cite{BlanchiniBook08,BlanchiniAuto99}, the invariance of $\C$ is equivalent to
	$$
	h(x)=0 \quad \Rightarrow \quad L_fh(x)\geq 0,
	$$
	which implies that $\alpha(0)\leq 0$. There always exists a class $\mathcal{K}$ function $\hat\alpha$ defined on $[0,\infty)$ that upper-bounds $\alpha$, yielding $\dot{h}(x)\geq -\hat\alpha(h(x))$ for all $x\in \C$.	
	This completes the proof.
\end{proof}

Propositions \ref{thm:GBF} and  \ref{propnecezbf} together show that a set $\C$ is forward invariant if, and only if, it admits a ZBF. Before addressing necessity for RBFs, a lemma is given. The shorthand notation $\dot{h}(x)$ is used for $L_fh(x)$, analogous to common usage for Lyapunov functions.

\begin{lemma}\label{lemiff}
	Consider the dynamical system \eqref{eqn:dynamicalsystem} and a nonempty, compact set $\C$ defined by \eqref{eqn:superlevelsetC}-\eqref{eqn:superlevelsetC3}
	for a continuously differentiable function $h$. If $\dot{h}(x)>0$ for all $x\in\partial\C$, then for each integer $k\geq 1$, there exists a constant $\gamma>0$ such that
	$$\dot{h}(x)\geq -\gamma h^k(x),\;\forall x\in\mathrm{Int}(\mathcal{C}).$$
\end{lemma}

%Because $\C$ is compact, local Lipschitz continuity of $L_fh$ is equivalent to Lipschitz continuity of $L_fh$ where the Lipschitz constant is assumed to be $L_1>0$. Because $\overline\C=\overline{\mathrm{Int}(\mathcal{C})}$ and $\C$ is nonempty,  there exists  some $\epsilon_0>0$ such that $\mathcal{Q}:=\{x|x\in\C, dist(x,\partial\C)\leq \epsilon_0\}\subsetneqq\C$. For any $x\in\mathcal{Q}$, there exists $x_0\in\partial\C$ such that $\|x-x_0\|\le \ep_0$ and  $|L_fh(x)-L_fh(x_0)|\le L_1 \|x-x_0\|$.  By Nagumo's theorem \cite{BlanchiniBook08}, the forward invariance of $\C$ is equivalent to $L_fh(x_0)\geq 0$. Thus, $L_fh(x)\ge L_fh(x_0)-L_1 \|x-x_0\|\ge -L_1 \|x-x_0\|$. On the other hand, ...(to be finished)

%$\forall x\in\mathrm{Int}(\mathcal{C}),h(x)>0$ while $\forall x\in\partial \C,h(x)=0$ and
%Therefore, for any $x\in\mathcal{Q}$, \mbox{$\L_fh(x)\geq -\gamma' h^k$} holds for some constant $\gamma'>0$.
\begin{proof}
	Because \mbox{$\C=\overline{\mathrm{Int}(\mathcal{C})}$} and $\C$ is nonempty, we have $\mathrm{Int}(\mathcal{C})\neq\emptyset$. Furthermore, because $\dot{h}(x)>0$ for all $x\in\partial\C$, by the continuity of $\dot{h}$, there exists $\ep_0>0$ such that $\dot{h}(x)>0$ for all $x\in\mathcal{Q}\subsetneqq\C$ where $\mathcal{Q}:=(B_{\ep_0}(0)\oplus \partial \C)\cap\mathrm{Int}(\mathcal{C})$ is an open set contained in $\mathrm{Int}(\mathcal{C})$.
	It follows that $\dot{h}(x)\geq -\gamma' h^k(x)$ holds for any $x\in\mathcal{Q}$ and any constant $\gamma'>0$, because the left hand side is non-negative and the right hand side is non-positive.
	
	Note that $\overline{\C\backslash \mathcal{Q}}$ is a compact subset because $\C$ is compact; moreover,
	$-\frac{\dot{h}}{h^k}$ is well-defined and continuous in $\overline{\C\backslash \mathcal{Q}}$. Hence, we can choose some constant $\gamma''\geq\max_{\{x|x\in\overline{\C\backslash \mathcal{Q}}\}}-\frac{\dot{h}(x)}{h^k(x)}$, such that
	$\dot{h}(x)\geq -\gamma'' h^k(x)$ holds for any \mbox{$x\in\overline{\C\backslash \mathcal{Q}}$}.
	
	Taking $\gamma=\gamma''$, we have $\dot{h}(x)\geq -\gamma h^k(x)$ for any
	$x\in\mathrm{Int}(\mathcal{C})$, which completes the proof.
\end{proof}

Based on the lemma, we have the following theorem.

\begin{theorem}\label{thm:CBFiff}
	Under the assumptions of Lemma \ref{lemiff},
	$B=\frac{1}{h}:\mathrm{Int}(\mathcal{C}) \to \R$ is a RBF and
	$h: \C \to \R$ is a ZBF for $\C$.
\end{theorem}
\begin{proof}
	Let $k=3$ in Lemma \ref{lemiff}. Then there exists $\gamma_1>0$ such that for all $x\in\mathrm{Int}(\mathcal{C})$, $\dot{h}\geq -\gamma_1 h^3$ holds, which implies that  $-\frac{\dot{h}}{h^2}\leq \gamma_1 h$ holds, or equivalently, $\dot{B}\leq \frac{\gamma_1}{B}$ holds. By Definition \ref{def:CBF}, $B=\frac{1}{h}$ is a RBF for $\C$.
	
	Let $k=1$ in Lemma \ref{lemiff}. Then there exists $\gamma_2>0$ such that for all $x\in\mathrm{Int}(\mathcal{C})$, $\dot{h}\geq -\gamma_2 h$ holds. By Definition \ref{def:barrierfunctions2}, $h$ is a ZBF defined on $\C$.
\end{proof}

\begin{remark}
	The assumption $\dot{h}(x)>0$ for all $x\in\partial\C$ is called contractivity in \cite{BlanchiniAuto99}, because the flow on the boundary of $\C$ points inward. Without the compactness assumption on $\C$, counterexamples to Thm.~\ref{thm:CBFiff} can be given. Consider a dynamical system on $\mathbb{R}^2$  given by $\dot{x}_1=-\frac{1}{2}x_2$, $\dot{x}_2=-x_1^3+1$.
	Define $\C=\{x|h(x)\geq 0\}$, where $h(x)=x_2-x_1^2$. Note that $\C$ is forward invariant because for any $x\in\partial\C$, $\dot{h}(x)=1> 0$. Clearly, $\C$ is not compact and $\dot{h}=\dot{x}_2-2x_1\dot{x}_1=x_1(x_2-x_1^2)+1$.
	For any $r>0$,
	$$\inf_{\{x|h(x)=r\}}\dot{h}(x)=\inf_{\{x|h(x)=r\}}x_1r+1=-\infty.$$
	Consequently, there cannot exist an extended $\mathcal{K}$ function $\alpha$ such that $\dot{h}\geq-\alpha(h)$, which implies that $h$ cannot be a ZBF for $\C$. Similarly, it is also impossible to find a class $\Kinfinity$ function $\alpha_3$ such that $-\frac{\dot{h}}{h^2}\leq \alpha_3(h)$ (resp. $-\frac{\dot{h}}{h(h+1)}\leq \alpha_3(h)$), which implies that \eqref{eqn:Bmotivation} (resp. \eqref{eqn:barrierfunc2}) cannot be a RBF for $\C$.
\end{remark}

The relationships of RBFs, ZBFs and  the set invariance are summarized in Fig.\ref{figbarrier2}. Note that while a ZBF leads to $\C$ being invariant,  when $\C$ is contractive, $\IntC$ is also invariant.

\section{Control Barrier Functions}
\label{sec:controlbarrierExtensions}
While barrier functions are important tools to verify invariance of a set, they cannot be directly used to design a controller enforcing invariance. By drawing inspiration on how Lyapunov functions were extended to control Lyapunov functions (by Sontag), we propose in this section a similar extension of barrier functions to control barrier functions (CBFs).  It is important to note that CBFs have been considered in the context of existing notions of barrier certificates \cite{Romd2014CDCunitiing,wieland2007constructive,prajna2007framework}.  The construction presented here differs due to the novel RBF condition \eqref{eqn:Bdotinequality} and the ZBF condition  \eqref{eqn:ZBFinequality},  which increases the available control inputs that satisfy the CBF condition.  Ultimately, the true usefulness of this will be seen when CBFs are unified with control Lyapunov functions through quadratic programs in Section \ref{sec:lyapunovfunction}.

\subsection{Reciprocal Control Barrier Functions}\label{sub:CBF}

Consider an affine control system
\begin{eqnarray}
	\label{eqn:controlsys}
	\dot{x} = f(x) + g(x) u,
\end{eqnarray}
with $f$ and $g$ locally Lipschitz, $x \in \R^n$ and $u \in U \subset \R^m$.  Later, we will be particularly interested in the case that $U$ can be expressed as a convex polytope,
\begin{equation}
	\label{eq:Upolytope}
	U=\{u\in\R^m|A_0u\leq b_0\},
\end{equation}
where $A_0$ is a $p\times m$ matrix and $b_0$ is a $p\times 1$ column vector of constants with $p$  some positive integer.

When the set $\mathrm{Int}(\mathcal{C})$ is not forward invariant under the natural dynamics of the system, $\dot{x} = f(x)$, how can a controller be specified that will ensure the invariance of $\mathrm{Int}(\mathcal{C})$? This motivates the following  definition.

%\gap
%\textbf{[JWG: removed differentiability of h, and added it to B]}
\begin{definition}\label{def:CBF}
	Consider the control system \eqref{eqn:controlsys} and the set $\mathcal{C} \subset \R^n$ defined by \eqref{eqn:superlevelsetC}-\eqref{eqn:superlevelsetC3} for a continuously differentiable  function $h$. A continuously differentiable function $B: \mathrm{Int}(\mathcal{C}) \to \R$ is called a \emph{reciprocal control barrier function (RCBF)} if there exist class $\Kinfinity$ functions $\alpha_1,\alpha_2,\alpha_3$ such that, for all $x \in \mathrm{Int}(\mathcal{C})$,
	\begin{align}
		\label{eq:inequality}
		&\quad\quad\quad \frac{1}{\alpha_1(h(x))}   \leq   B(x)    \leq \frac{1}{\alpha_2(h(x))}\\
		\label{eq:Binf}
		& \inf_{u \in U}  \left[ L_f B(x) + L_g B(x) u -\alpha_3(h(x))  \right] \leq 0.
	\end{align}
	The RCBF $B$ is said to be locally Lipschitz continuous if $\alpha_3$ and $\frac{\partial B}{\partial x}$ are both locally Lipschitz continuous.
\end{definition}

%\begin{remark}
%%From \eqref{eq:inequality}, $\frac{1}{ B(x)} \le \alpha_1(h(x))$.
%Comparing \eqref{eq:Binf} in the above definition with Def. \ref{def:barrierfunctions}, we are effectively taking $\alpha_3(x) = \gamma \alpha_1(x)$. While we find this convenient in applications, the more general bound
%$$ \inf_{u \in U}  \left[ L_f B(x) + L_g B(x) u - \alpha_3(h(x)) \right] \leq 0$$
%yields the same controlled forward invariance result.
%\end{remark}

Given a RCBF $B$, for all $x\in\mathrm{Int}(\mathcal{C})$, define the set
% of all control values that satisfy \eqref{eq:Binf}:
\begin{equation*}
	\Krcbf(x) =  \{ u \in U : L_f B(x) + L_g B(x) u -  \alpha_3(h(x)) \leq 0\}. \nonumber
\end{equation*}
%\begin{eqnarray}
%\lefteqn{\Kcbf(x) = } \nonumber\\
%&& \{ u \in U : L_f B(x) + L_g B(x) u - \frac{\gamma}{ B(x)} \leq 0\}. \nonumber
%\end{eqnarray}
Considering control values in this set allows us to guarantee the forward invariance of $\C$ via the following straightforward application of Theorem \ref{thm:main}.

%\gap

%\begin{corollary}
%\label{cor:cbf}
%{\it Given a set $\mathcal{C} \subset \R^n$ be defined by \eqref{eqn:superlevelsetC}-\eqref{eqn:superlevelsetC2} with $h: \R^n \to \R$ continuously differentiable, for any Lipschitz continuous controller $u(x) \in \Kcbf(x)$ for the system \eqref{eqn:controlsys} renders the set $\C$ forward invariant.}
%\end{corollary}

\begin{corollary}
	\label{cor:cbf}
	Consider a set $\mathcal{C} \subset \R^n$ be defined by \eqref{eqn:superlevelsetC}-\eqref{eqn:superlevelsetC3} and let $B$ be an associated RCBF for the system \eqref{eqn:controlsys}. Then any locally Lipschitz continuous controller $u: \mathrm{Int}(\C) \to U$ such that $u(x) \in \Krcbf(x)$ will render the set $\mathrm{Int}(\mathcal{C})$ forward invariant.
\end{corollary}

\subsection{Zeroing Control Barrier Functions}

%Consider the affine control system \eqref{eqn:controlsys} and the input set $U$ defined in \eqref{eq:Upolytope}.
Def. \ref{def:extend} for ZBFs leads to the second type of control barrier function.
%Using Definition \ref{def:extend}, we can define the notion of a \emph{zeroing control barrier function} as follows.

\begin{definition}\label{dfn:newcbf}
	Given a set $\mathcal{C} \subset \R^n$ defined by \eqref{eqn:superlevelsetC}-\eqref{eqn:superlevelsetC3} for a continuously differentiable function $h: \R^n \to \R$,  {\color{black}the function $h$ is called} a \emph{zeroing control barrier function (ZCBF)} defined on   set $\mathcal{D}$ with $\C\subseteq\mathcal{D}\subset \R^n$, if  there exists an extended class $\Kinfinity$ function $\alpha$ such that
	\begin{align}\label{ineq:ZCBF}
		& \sup_{u \in U}  \left[ L_f h(x) + L_g h(x) u + \alpha(h(x))\right] \geq 0,\;\forall x \in \mathcal{D}.
	\end{align}
	The ZCBF $h$ is said to be locally Lipschitz continuous if $\alpha$ and the derivative of $h$ are both locally Lipschitz continuous.
\end{definition}
%{\color{red}unify \eqref{eq:Binf} and \eqref{ineq:ZCBF}}

%If $U=\R^m$ and $L_gh(x) \ne 0$ for $x\in\mathcal{D}$, then the function $h$ is always a ZCBF.

Given a ZCBF $h$, for all $x\in\D$ define the set
% of all control values that satisfy \eqref{eq:Binf}:
\begin{equation}\label{zcbfinputset}
	\Kzcbf(x) =  \{ u \in U : L_f h(x) + L_g h(x) u + \alpha(h(x)) \geq 0\}. \nonumber
\end{equation}

Similar to Corollary \ref{cor:cbf}, the following result guarantees the forward invariance of $\mathcal{C}$.

\begin{corollary}\label{cor:zbf}
	Given a set $\mathcal{C} \subset \R^n$ defined by \eqref{eqn:superlevelsetC}-\eqref{eqn:superlevelsetC3} for a continuously differentiable function $h$, if $h$ is a ZCBF on $\D$, then any Lipschitz continuous controller $u: \mathcal{D} \to U$ such that $u(x) \in \Kzcbf(x)$ will render the set $\mathcal{C}$ forward invariant.
\end{corollary}

\begin{remark}
	%	{\bf [Remark Added by Xiangru]}
	Note that control \mbox{$u(x) \in \Krcbf(x)$} (or \mbox{$u(x) \in \Kzcbf(x)$}) will not necessarily render the closed-loop system of  \eqref{eqn:controlsys} forward complete, but only ensures that if  $x_0 \in \mathrm{Int}(\mathcal{C})$, then  $x(t) \in \mathrm{Int}(\mathcal{C})$ for all $t \in I_u(x_0)$. Here, $I_u(x_0)$ is the maximal time interval for the closed-loop system of \eqref{eqn:controlsys} with control $u(x) \in \Krcbf(x)$ (resp.  $u(x) \in \Kzcbf(x)$).
\end{remark}

%\begin{remark}
%{\color{blue} Mention that there is no difference using RCBF or ZCBF when $u$ has no constraints?}
%\end{remark}
\subsection{Higher Relative Degree}

In the preceding two subsections, if the function $h$ has a relative degree greater than 1, then $L_gh=0$ and the set $\Krcbf(x)$ or $\Kzcbf(x)$ trivially equals to $U$ or the empty set. When $h$ has a relative degree $r\ge2$,  the following proposition shows how to design a RCBF for $\C$.

\begin{proposition} Consider the control system \eqref{eqn:controlsys} with \mbox{$U=\R^m$}. Consider also a set $\C \subset \R^n$ defined by \eqref{eqn:superlevelsetC}-\eqref{eqn:superlevelsetC3} for a function $h$ with relative degree $r\ge2$, namely, $h$ is $r$-times continuously differentiable and $\forall~x\in \IntC$,  $L_gL_f^{k}h(x) =0$, for $0\le k \le r-2$, and  $L_gL_f^{(r-1)}h(x) \neq 0$. Then for any constant $H_{max}> 0$  and continuously differentiable function $H: \R\rightarrow \R_0^+$ satisfying
	\begin{align}
		&(i)\; 0 \le H(\lambda)\le H_{max},\; \forall~\lambda\in \R, \label{extpro1}\\
		&(ii)\; \frac{dH(\lambda)}{d \lambda}\neq 0,\; \forall~\lambda\in \R, \label{extpro3}
		%&(ii)\; \text{if}\; \frac{dH}{d \lambda}|_{\lambda=\lambda^*}=0\;\text{then}\;\lambda^*=0,\label{extpro3}
	\end{align}
	the function $B_r: \IntC \to \R_0^+$ defined by
	$$B_r:=\frac{1}{h}+H\circ L_f^{(r-1)}h$$
	is a RCBF.
\end{proposition}

\begin{proof}
	%Suppose that $h$ has relative degree $r$ with $r\geq 2$ in $\C$, in the sense that $L_gL_f^{(i)}h(x) \equiv 0$, $i=0,1,...,r-2$, and $L_gL_f^{(r-1)}h(x) \neq 0$ for all $x\in \mathrm{Int}(\mathcal{C})$. Suppose that $U=\R^m$.  Suppose furthermore that $B_1$ is constructed from $h$ via \eqref{eqn:Bmotivation} and satisfies condition \eqref{eqn:Binequality}.
	%Then,
	%\begin{equation}\label{BarrierRelDegree2}
	%B_r=B_1+H(L_f^{(r-1)}h)
	%\end{equation}
	%is a RCBF for $\C$. Indeed,
	For all $x \in \mathrm{Int}(\mathcal{C})$,
	%\begin{align*}
	% \frac{1}{h(x)}&\leq B_r(x)\leq \frac{1}{h(x)}+ H_{max},\\
	% \frac{1}{\alpha_1(h(x))}  & \leq   B_1(x)    \leq \frac{1}{\alpha_2(h(x))},
	%\end{align*}
	$$
	\frac{1}{h(x)}\leq B_r(x)\leq \frac{1}{h(x)}+ H_{max}$$
	and thus
	$$
	\frac{1}{\alpha_1(h(x))}  \leq   B_r(x)    \leq \frac{1}{\alpha_2(h(x))},$$
	where $\alpha_1(\xi):=\xi$ and
	$$\alpha_2(\xi) := \begin{cases} 0 & \text{If}~\xi = 0 \\
	\frac{1}{ \frac{1}{\xi} + H_{max} } & \text{If}~ \xi >0
	\end{cases}  $$
	are both class $\Kinfinity$ functions. Thus, condition \eqref{eq:inequality} is satisfied.
	By the chain rule,
	%\dot{B}_r=\frac{dB_1}{dh}\dot{h}+\frac{dH}{d \lambda}(L_fh)\left(L_f^{r}h+L_gL_f^{(r-1)}hu\right).\label{dotB2}
	%\end{align}
	\begin{align}
		L_g{B}_r=\left( \frac{dH}{d \lambda}\circ L_f^{(r-1)}h\right) \left(L_gL_f^{(r-1)}h\right).\label{dotB2}
	\end{align}
	Because $h$ has relative degree $r$ and \eqref{extpro3} holds, it follows that $B_r$  has relative degree one. Therefore, for any class $\Kinfinity$ function $\alpha_3$, and for any $x\in \IntC$, there exists $u\in \R^m$ such that $L_f B_r(x)+ L_g B_r(x) u \leq \alpha_3(1/B_r(x))$, and thus condition \eqref{eq:Binf} holds. Therefore, $B_r$ is a RCBF.
\end{proof}
%More generally, if $h$ has (uniform) relative degree $r\ge2$, the same construction with $B_1+H(L_f^{(r-1)})$ is a RCBF for $\C$.

An example for $H$ is \mbox{$H(\lambda)={\rm atan}(\lambda)+\frac{\pi}{2}\in(0,\frac{\pi}{2})$}, where $\frac{dH}{d \lambda}=\frac{1}{1+\lambda^2}\neq 0$ for any $\lambda$. Another means to construct RCBFs for $h$ with relative degree $r\ge2$ is given in \cite{hsubackstepping}, where a backstepping-inspired method for its construction is provided.

\begin{remark}
	For ZCBF $h$ with relative degree $r\ge 2$, replace  \eqref{extpro1} by there exists $H_{min}>0$, $H_{max}>0$ and
	%the result is almost the same. Replace  \eqref{extpro1} by there exists $H_{min}>0$ and
	\begin{equation}
		\label{extpro1b}
		H_{min} \le H(\lambda)\le H_{max},\; \forall~\lambda\in \R.
	\end{equation}
	Then for any $H(\lambda)$ satisfying \eqref{extpro1b} and \eqref{extpro3},  the function \mbox{$(H\circ L_f^{(r-1)})\cdot h$} is a ZCBF defined on $\C$. See also \cite{nguyen2016exponential} for an alternative approach.
	%An example would be  \mbox{$H(\lambda)={\rm atan}(\lambda)+\pi$}.
\end{remark}

\begin{remark}	
	Note that if $U\neq\R^m$, i.e., there are constraints on the input $u$,
	then the construction shown above for higher relative degree $h$ may no longer be valid. Designing CBFs in this case remains an open question.
\end{remark}

%Note that argument above can not be extended directly to ZCBFs with $h$ relative degree greater than $1$.

%%
%%
\section{QPs for Mediating Safety and Performance}
\label{sec:lyapunovfunction}
In this section, we address the following question: how to select, among the control inputs that enforce the safety requirement, an input that also enforces liveness/stability?
%In this section, the key idea is that among control inputs that are compatible with safety, we want to select an input that most closely achieves performance.
We begin with a brief overview of exponentially stabilizing control Lyapunov functions in the context of nonlinear systems.  This formulation naturally leads to a quadratic program (QP) that allows for the unification of control Lyapunov functions for performance and control barrier functions for safety.

In the following, the dynamics of the system are given by a nonlinear affine control system of the form
\begin{eqnarray*}
	\left( \begin{array}{c} \dot{x}_1 \\\dot{x}_2\end{array}\right)&=& \underbrace{\left( \begin{array}{c} f_1(x_1,x_2) \\ f_2(x_1,x_2)\end{array}\right)}_{f(x)} + \underbrace{\left( \begin{array}{c} g_1(x_1,x_2) \\ 0\end{array}\right)}_{g(x)} u
\end{eqnarray*}
%\begin{align}
%%\label{eqn:fgdynamics}
%\dot{x}_1 & = f_1(x_1,x_2) + g_1(x_1,x_2) u \nonumber\\
%\dot{x}_2 & = f_2(x_1,x_2) \nonumber
%\end{align}
where $x_1 \in X\subset\R^{n_1}$ are controlled (or output) states, \mbox{$x_2 \in Z\subset\R^{n_2}$} are the uncontrolled states, with $n_1+n_2=n$, and $U\subset\R^m$ is the set of admissible control values for $u$.  In addition, we assume that $f_1(0,x_2) = 0$, i.e., that the zero dynamics surface $Z$ defined by $x_1 = 0$ with dynamics given by $\dot{x}_2 = f_2(0,x_2)$ is invariant, and we assume adequate smoothness assumptions on the dynamics so that solutions are well defined.

%\vspace{-.5cm}

\subsection{Control Lyapunov Functions}

%\textbf{[JWG: defined  $x=(x_1,x_2)$, and replaced the old $x$ by $x_1$ and old $z$ by $x_2$. Defined $V: X \times Z \to \mathbb{R}$. Trying to make the notation consistent with the barrier function. I think this is OK, but please check.]}

\begin{definition}\cite{AmGaGrSr:TAC12}
	\label{def:resclf}
	A continuously differentiable function  \mbox{$V : X \times Z\to \mathbb{R}$} is an exponentially stabilizing control Lyapunov function (ES-CLF) if there exist positive constants $c_1, c_2, c_3 > 0$ such that for all $x=(x_1,x_2) \in X \times Z$, the following inequalities hold,
	\begin{eqnarray}
		\label{eq:Vineq}
		&c_1 \| x_1\|^2 \leq V(x) \leq c_2 \|x_1 \| ^2,  &\\
		\label{eq:Vinf}
		&\inf_{u \in U} \left[ L_f V(x) + L_g V(x) u +  c_3 V(x) \right] \leq 0.&
	\end{eqnarray}
\end{definition}

%\newsec{Min-Norm Controller.}
The existence of an ES-CLF yields a family of controllers that exponentially stabilize the system to the zero dynamics \cite{FK:Book}.  In particular, consider the set
\begin{equation*}
	\Kclf(x) =  \{ u \in U : L_f V(x) + L_g V(x) u + c_3 V(x) \leq 0\}.
\end{equation*}
%\begin{eqnarray}
%\lefteqn{\Kclf(x,z) = } \nonumber\\
%&& \{ u \in U : L_f V(x,z) + L_g V(x,z) u + c_3 V(x) \leq 0\}. \nonumber
%\end{eqnarray}
It follows that a locally Lipschitz controller $u:X \times Z \to U$ satisfies
\begin{equation*}
	%\label{eqn:convergence}
	u(x)  \in  \Kclf(x)  \Rightarrow \| x_1(t) \| \leq \sqrt{\frac{c_2}{c_1}} e^{- \frac{ c_3}{2} t} \| x_1(0) \|.
\end{equation*}

When  $U = \R^m$, Freeman and Kokotovic introduced the \textit{min-norm controller}, $u^\ast(x)$, defined pointwise as the element of $\Kclf(x)$ having minimum Euclidean norm \cite{FreemanSIAM96}.
%Extensions can be found in \cite{yuqing2007generalized}, with applications to satellite control in \cite{horri2011energy}.
The min-norm controller can be interpreted  as the solution of a quadratic program (QP). Importantly, by using the QP formulation, it is straightforward to include bounds on the control values \cite{Spong86,Kevin2015torque}, such as those given in \eqref{eq:Upolytope}, namely
\begin{align}
	\label{eqn:QPCLF}
	u^*(x) =   & \underset{u \in \R^{m}}{\operatorname{argmin}}   \quad \frac{1}{2}
	u^\top u\\
	\mathrm{s.t.} &  \quad L_fV(x) + L_gV(x) u \le - c_3 V(x) \nonumber\\
	%\label{eqn:CBFcon}\tag{CBF}
	& \quad A_0u\leq b_0.\nonumber
\end{align}
The QP-form of the controllers have been executed in real-time to achieve bipedal walking \cite{Kevin2015torque,YouTubeExp1} on a human-sized robot and on scale cars \cite{Aakar2015experiment}, with sample rates of 200 Hz to 1 kHz.

\subsection{Combining CLFs and CBFs via QPs}

A distinct advantage of the QP perspective is that it allows for the unification of control performance objectives (represented by CLFs) subject to the trajectories belonging to desired ''safe'' sets (as dictated by CBFs). By relaxing the constraint represented by the CLF condition \eqref{eq:Vinf}, and adjusting the weight on the relaxation parameter, the QP can mediate the tradeoff between performance and safety, with the safety being guaranteed.

%\textbf{[JWG: I fixed the notation so that we do not need ${\bf x}$. I replaced $\C$ with $\IntC$. We need to talk about constraints $U$ to the QP.]}

Specifically, given a RCBF $B$ associated with a set $\C$ defined by \eqref{eqn:superlevelsetC}-\eqref{eqn:superlevelsetC3} and an ES-CLF $V$, they can be combined into a single controller through the use of a QP of the following form\footnote{In the following sections, only RCBFs are used to formulate the QPs; however, QPs incorporating ZCBFs can be formulated in a similar way \cite{Xu2015ADHS}.}
\begin{align}
	\label{eqn:QPCLFCBF}\tag{CLF-CBF QP}
	\ubold^*(\xbar) =   & \underset{\ubold = (u,\delta)\in \R^{m}\times \R}{\operatorname{argmin}}   \quad \frac{1}{2}
	\ubold^\top H(\xbar) \ubold + F(\xbar)^\top \ubold \\
	%\label{eqn:CLFcon}\tag{CLF}
	\mathrm{s.t.} &  \quad L_fV(\xbar) + L_gV(\xbar) u +c_3 V(\xbar) -\delta \leq 0, \label{QP:CLFconstraint}\\
	%\label{eqn:CBFcon}\tag{CBF}
	&  \quad L_f B(\xbar) + L_g B(\xbar) u -\alpha(h(\xbar))\leq 0,\label{QP:CBFconstraint}
	% &\quad A_0u\leq b_0,\nonumber
\end{align}
where $c_3>0$ is a constant, $\alpha$ belongs to class $\mathcal{K}$,  $H(\xbar) \in \R^{(m +1) \times (m+1)}$ is positive definite, and $F(\xbar) \in \R^{m+1}$.

%
%Specifically, given a CLF $V$ and a RCBF $B$ with relative degree 1 in $\C$, the two ``specifications'' are  combined via the following parameterized quadratic program
%\begin{align}
%\mathcal{P}(x):&\quad \forall~x\in\C,~ \nonumber \\
%&\quad \ubold^*(x)=  \underset{\ubold = \left[ u^\top,\delta \right]^\top\in\R^{m+1}}{\operatorname{argmin}}
%\ubold^\top  \ubold \nonumber\\
%%&\quad \ubold^*(x)=  \underset{\ubold = \left[ \begin{array}{c} u \\ \delta \end{array} \right]\in\R^{m+1}}{\operatorname{argmin}}
%%\ubold^\top  \ubold \nonumber\\
%&\quad \mathrm{s.t.}  \;  L_g V(x) u +L_f V(x)- \delta\leq 0,  \label{QP:ESCLF}\\
%&\quad \quad\;\; L_g h(x) u+L_f h(x) +\alpha(h(x))   \geq 0 ,\label{QP:ZCBF}
%\end{align}
%where $u\in\R^m$ is the control input, {\color{black}$\delta$} is a relaxation parameter, constraint \eqref{QP:ZCBF} is the ZCBF condition and constraint \eqref{QP:ESCLF} is the CLF condition.
%
%The following theorem is the main result of this subsection.
% Its proof can be found in \cite{Xu2015ADHS} and is omitted here.

The following theorem\footnote{Note that while this theorem is established for ES-CLFs in this paper, the same results hold for classically defined CLFs as in \cite{FK:Book}.} provides a sufficient condition for $\ubold^*(\xbar)$ to be locally Lipschitz continuous in $\IntC$, thereby guaranteeing local existence and uniqueness of solutions to the closed-loop system, and the applicability of Corollaries \ref{cor:cbf} and \ref{cor:zbf}.

\begin{theorem}\label{thm:LocLip}
	Suppose that the following functions are all locally Lipschitz: the vector fields $f$ and $g$ in the control system \eqref{eqn:controlsys}, the gradients of the RCBF $B$ and CLF $V$, as well as the cost function terms $H(\xbar)$ and $F(\xbar)$ in \eqref{eqn:QPCLFCBF}. Suppose furthermore that the relative degree one condition, $L_g B(\xbar) \neq 0$ for all $x \in \IntC$, holds.
	Then the solution, $\ubold^*(\xbar)$, of \eqref{eqn:QPCLFCBF} is locally Lipschitz continuous for $\xbar \in \IntC$. Moreover, a closed-form expression can be given for $\ubold^*(\xbar)$.
\end{theorem}
%%Proof updated by JWG on 4 April 2016
\begin{proof}
	%\textbf{[JWG is updating the proof. Lyapunovfunctions\_v03.tex has the original text.]}
	The proof is based on \cite{LuenbergerOptimizationVectorSpaceMethods}[Ch.~3], which as a special case includes minimization of a quadratic cost function subject to affine inequality constraints.
	
	Define
	\begin{align*}
		& y_1(\xbar)=[L_g V(\xbar),-1]^\top,\;p_1(\xbar)= -L_f V(\xbar)-c_3V(x),\\
		& y_2(\xbar)=[L_g B(\xbar),0]^\top,\;p_2(\xbar)=-L_f B(\xbar) +\alpha(h(\xbar)),
	\end{align*}
	and note that for all $\xbar\in \IntC$,  $y_1(\xbar)$ and $y_2(\xbar)$ are linearly independent in $\mathbb{R}^{m+1}$.
	Because $H(x)$ is locally Lipschitz continuous and positive definite, its inverse exists and is locally Lipschitz continuous. Define
	$$
	\begin{bmatrix} \bar{y}_1(\xbar), \bar{y}_2(\xbar) \end{bmatrix} = H(\xbar)^{-1} \begin{bmatrix} {y}_1(\xbar), {y}_2(\xbar) \end{bmatrix},
	$$
	$$
	\begin{bmatrix} \bar{p}_1(\xbar) \\ \bar{p}_2(\xbar) \end{bmatrix} =  \begin{bmatrix} {p}_1(\xbar) \\ {p}_2(\xbar) \end{bmatrix} -  \begin{bmatrix} {y}_1(\xbar)^\top \\  {y}_2(\xbar)^\top \end{bmatrix} \bar{\ubold}(\xbar),
	$$
	and
	\begin{align*}
		%\label{eqn:Urelated2V}
		\bar{\ubold}(\xbar)&:= -H(\xbar)^{-1}F(\xbar)\\
		\vbold &:=  \ubold-\bar{\ubold}(\xbar).
	\end{align*}
	Finally, let$\left<\cdot, \cdot \right>$ define an inner product on $\mathbb{R}^{ m+1}$ with weight matrix $H(x)$ so that
	$
	\left< \vbold,  \vbold \right>: = \vbold^\top H(\xbar)\vbold.
	$

	The optimization problem \eqref{eqn:QPCLFCBF} is then equivalent to
	\begin{align}
		\vbold^*(\xbar) =   & \underset{\vbold \in \R^{m+1}}{\operatorname{argmin}}  \left<\vbold, \vbold \right> \label{eqn:QPCLFCBF:equivalent}\\
		\mathrm{s.t.} &  \left<\bar{y}_1(\xbar),  \vbold \right>  \le \bar{p}_1(\xbar), \nonumber \\
		& \left<\bar{y}_2(\xbar),  \vbold \right> \le \bar{p}_2(\xbar), \nonumber
	\end{align}
	with
	\begin{align}
		\label{eqn:Urelated2V}
		\ubold^*(\xbar)&=\vbold^*(\xbar) + \bar{\ubold}(\xbar).
	\end{align}
	
	From  \cite{LuenbergerOptimizationVectorSpaceMethods}[Ch.~3], the solution to \eqref{eqn:QPCLFCBF:equivalent} is computed as follows. Let $G(\xbar)=[G_{ij}(\xbar)]=[\langle \bar{y}_i(\xbar), \bar{y}_j(\xbar)\rangle]$, $i,j=1,2$ be the Gram matrix. Due to the linear independence of $\{\bar{y}_1(\xbar), \bar{y}_2(\xbar)\}$,  $G(\xbar)$ is positive definite. The unique solution to \eqref{eqn:QPCLFCBF:equivalent} is
	\begin{equation}
		\label{eqn:Lambdas}
		\vbold^*(\xbar) = \lambda_1(\xbar)  \bar{y}_1(\xbar) +  \lambda_2(\xbar)  \bar{y}_2(\xbar) ,
	\end{equation}
	where $\lambda(\xbar)=[\lambda_1(\xbar), \lambda_2(\xbar)]^\top$ is the unique solution to
	\begin{align}
		\label{eqn:LuenbrgerFormulation}
		G(\xbar) \lambda(\xbar) & \le \bar{p}(\xbar), \nonumber \\
		\lambda(\xbar) & \le 0,\\
		[G(\xbar) \lambda(\xbar)]_i & < \bar{p}_i(\xbar) ~ \Rightarrow~\lambda_i(\xbar)=0, \nonumber
	\end{align}
	where $[\cdot]_i$ denotes the $i$-th row of the quantity in brackets, and the inequalities hold componentwise.
	Because $G(\xbar) $ is $2 \times 2$, a closed form solution can be given. Define the Lipschitz continuous function
	$$
	\omega(r)=\left\{\begin{array}{rl}
	0,& \mbox{if}\;r>0, ~\\
	r,&\mbox{if}\; r\leq 0.
	\end{array} r\in \mathbb{R}. \right.
	$$
	For $\xbar\in\IntC$,
	%%and dropping the argument $\xbar$ for compactness of notation,
	$\lambda_1 ,  \lambda_2$ can be expressed in closed form as\newline
	\noindent \textbf{If:} $G_{21}(\xbar) \omega (\bar{p}_2(\xbar))-G_{22}(\xbar)\bar{p}_1(\xbar)<0,$
	\begin{align}
		\label{eqn:IfCondition}
		\left[  \begin{array}{c}
			\lambda_1(\xbar) \\
			\lambda_2(\xbar)
		\end{array} \right] = \left[  \begin{array}{c}
		0\\
		\frac{\omega(\bar{p}_2(\xbar))}{G_{22}(\xbar)}
	\end{array} \right],
\end{align}
\noindent \textbf{Else if:} $G_{12}(\xbar) \omega (\bar{p}_1(\xbar))-G_{11}(\xbar)\bar{p}_2(\xbar)<0,$
\begin{align}
	\label{eqn:ElseIfCondition}
	\left[  \begin{array}{c}
		\lambda_1(\xbar) \\
		\lambda_2(\xbar)
	\end{array} \right] = \left[  \begin{array}{c}
	\frac{\omega(\bar{p}_1(\xbar))}{G_{11}(\xbar)} \\
	0
\end{array} \right],
\end{align}

\noindent \textbf{Otherwise:}
\begin{align}
	\label{eqn:OtherwiseCondition}
	\left[  \begin{array}{c}
		\lambda_1(\xbar) \medskip \\
		\lambda_2(\xbar)
	\end{array} \right] = \left[  \begin{array}{c}
	\frac{\omega(G_{22}(\xbar)\bar{p}_1)(\xbar)-G_{21}(\xbar)\bar{p}_2(\xbar))}
	{G_{11}(\xbar)G_{22}(\xbar)-G_{12}(\xbar)G_{21}(\xbar)} \medskip \\
	\frac{\omega(G_{11}(\xbar)\bar{p}_2(\xbar)-G_{12}(\xbar)\bar{p}_1(\xbar))}
	{G_{11}(\xbar)G_{22}(\xbar)-G_{12}(\xbar)G_{21}(\xbar)}
\end{array} \right].
\end{align}

Because the Gram matrix is positive definite, for all $\xbar\in\IntC$, $G_{11}(\xbar)G_{22}(\xbar)-G_{12}(\xbar)G_{21}(\xbar)>0$. Using standard properties for the composition and product of locally Lipschitz continuous functions, each of the expressions in \eqref{eqn:IfCondition} -\eqref{eqn:OtherwiseCondition} is locally Lipschitz continuous on $\IntC$. Hence, the functions $\lambda_1(\xbar)$ and $\lambda_2(\xbar)$ are locally Lipschitz on each domain of definition and have well defined limits on the boundaries of their domains of definition relative to $\IntC$. If these limits agree at any point $\xbar$ that is common to more than one boundary, then $\lambda_1(\xbar)$ and $\lambda_2(\xbar)$ are locally Lipshitz continuous on $\IntC$. However, the limits are solutions to  \eqref{eqn:LuenbrgerFormulation}, and solutions to \eqref{eqn:LuenbrgerFormulation} are unique  \cite{LuenbergerOptimizationVectorSpaceMethods}. Hence the limits agree at common points of their boundary\footnote{As an example, the only non-zero solutions of \eqref{eqn:LuenbrgerFormulation} occur when $p_2(\xbar)<0$, in which case,
	$G_{21}(\xbar)p_2(\xbar)-G_{22}(\xbar)p_1(\xbar)=0,$
	and  therefore \eqref{eqn:OtherwiseCondition} reduces to \eqref{eqn:IfCondition}. The other cases are similar.} (relative to $\IntC$) and the proof is complete.
\end{proof}

%\begin{remark}\label{rem:Correction2ADHS} The expressions \eqref{eqn:Lambdas} through \eqref{eqn:OtherwiseCondition} should be applied to correct the proof of \cite{Xu2015ADHS}[Thm.~8]. \textbf{[JWG: To be expanded.]}
%\end{remark}

%While the set $\C$ will be rendered forward invariant, the control objective may not necessarily be achieved.
If the control objective and the barrier function do not conflict, such as when the zero dynamics surface of the CLF has a non-empty intersection with the safe set, an appropriate choice of weights results in a solution of the QP with $\delta\approx 0$  \cite{Xu2015ADHS}. The mediation of safety and performance will be illustrated in the context of the adaptive cruise control and lane keeping problems in the following sections. The examples will also provide explicit control barrier functions that respect constraints on the inputs, such as those given in \eqref{eq:Upolytope}. In particular, the examples will add a constraint of the form
\begin{equation}
	\label{eqn:InputBoundForQP}
	A_0u - b_0 \le 0
\end{equation}
to the QP, in addition to \eqref{QP:CLFconstraint} and \eqref{QP:CBFconstraint}.
By construction, at each point of the safe set, there will exist a solution of the QP satisfying all three constraints. The Lipschitz continuity of the QP with the additional constraint \eqref{eqn:InputBoundForQP} on the inputs, however, is not currently assured.

\section{Two Automotive Safety Problems via QPs}
\label{sec:accviaqp}
In this section, we use Adaptive Cruise Control (ACC) and Lane Keeping (LK) to illustrate the power of a CLF-CBF-based QP to meet a performance objective, subject to a safety requirement.

\subsection{Adaptive Cruise Control Via QPs}
\label{sec:ACCencodingConstraints}
A vehicle equipped with ACC seeks to converge to and maintain a fixed cruising speed, as with a common cruise control system. Converging to and maintaining fixed speed is naturally expressed as asymptotic stabilization of a set. With ACC, the vehicle must in addition guarantee a safety condition, namely, when a slower moving vehicle is encountered, the controller must automatically reduce vehicle speed to maintain a guaranteed lower bound on time headway or following distance, where the distance to the leading vehicle is determined with an onboard radar. When the leading car speeds up or leaves the lane, and there is \textit{no longer a conflict between safety and desired cruising speed}, the adaptive cruise controller automatically increases vehicle speed. The time-headway safety condition is naturally expressible as a control barrier function. Because relaxation is used to make the stability objective a soft constraint in the QP, while safety is maintained as a hard constraint, safety and stability do not need to be simultaneously satisfiable. In contrast, the approach of \cite{Romd2014CDCunitiing} for combining CBFs and CLFs is only applicable when the two objectives can be simultaneously met. Simulations in Sect.~\ref{sec:simulation} will illustrate how the QP-based solution to ACC automatically adjusts vehicle speed under various traffic conditions.

\subsubsection{ACC problem setup}
We begin by setting up the dynamics of the problem based upon \cite{ioannou1993autonomous} and \cite{liang2000string}, which assume that the \textit{lead} and \text{following} vehicles are modeled as point-masses moving in a straight line. The following vehicle is the one equipped with ACC while the lead vehicle acts as a disturbance to the following vehicle's objective of cruising at a given constant speed.

The dynamics of the system can be compactly expressed as
\begin{align}\label{eqn:fgdynamics}
	\dot{x} &= \underbrace{\left( \begin{array}{c} -\Fr/M\\a_L\\x_2-x_1\end{array}\right)}_{f(x)} + \underbrace{\left( \begin{array}{c} 1/M\\0\\0\end{array}\right)}_{g(x)} u.
\end{align}
Here, $x = (x_1,x_2,x_3):=(v_f,v_l,D)$ where $v_f$ and $v_l$ are the velocity of the following and leading vehicle  (in $m/s$), respectively, $D$ is the distance between the two vehicles (in $m$); $M$ is the mass of the following vehicle (in $kg$);  $\Fr(x)=f_0 + f_1 v_f + f_2 v_f^2$ is the aerodynamic drag (in N) with constants $f_0$, $f_1$ and $f_2$ determined empirically; $a_L\in[-a_lg,a'_lg]$ is the overall acceleration/deceleration of the lead vehicle (in $m/s^2$) with $a_l,a_l'$ fractions of the gravitational constant $g$ for deceleration and acceleration, respectively; $u \in U\subset\R$, the control input of the following car, is wheel force (in N). Initially, we will suppose that the control input is unbounded, that is, $U=\R$, and later, we address realistic bounds on wheel force.

Given the model \eqref{eqn:fgdynamics}, we next present two constraints that are necessary in the context of ACC.

\newsec{Soft Constraints.} In the context of ACC, when adequate headway is assured, the goal is to achieve a desired speed, $v_d$.  In other words,
\begin{align}
	\tag{SC}
	\label{eqn:controlobj}
	\mathrm{Performance:} \quad  \lim_{t \to \infty} v_f(t)=v_d.
\end{align}
This translates into a soft constraint since this speed should only be achieved in the case when safety can be assured. In terms of a candidate CLF, the soft constraint \eqref{eqn:controlobj} can be written
\begin{align}
	\mathrm{Performance:} \quad  V(x): = (v_f- v_d)^2.\nonumber
\end{align}
Straightforward calculations given in \cite{aaroncbfcdc14} show that for any $c>0$, the following inequality holds,
$$
\inf_{u \in \R} \left[ L_f V(x) + L_g V(x) u + c V(x) \right] \leq 0,
$$
verifying that $V$ is a valid CLF.

\newsec{Hard Constraints.}  These represent constraints that must not be violated under any condition.  For ACC, this is simply the constraint: {\it ``keep a safe distance from the car in front of you''}.  There are numerous formulations of this concept including Time Headway and Time to Collision \cite{Vogel}.  In the context of this paper, to start with a simple formulation, we express this constraint as
\begin{align}
	\label{eqn:hard_constraint}
	\tag{HC}
	\mathrm{Safety:} \quad D/v_f \geq \tau_{d},
\end{align}
where $\tau_{d}$ is the desired time headway.\footnote{
	A general rule stated in \cite{Vogel} is that the minimum distance between two cars is ``half the speedometer''.  This translates into the hard constraint as $D\geq 1.8 v_f$ with $\tau_{d}=1.8$.}

The constraint \eqref{eqn:hard_constraint} can be rewritten as
\begin{align}
	\label{eqn:ACC_spec_constraint2}
	%\tag{HC1}
	D - \tau_{d} v_f\geq 0,
\end{align}
for the dynamics \eqref{eqn:fgdynamics}. Correspondingly, we consider the function $h(x) =D- \tau_{d}v_f$, which yields the admissible set $\C$ as defined in \eqref{eqn:superlevelsetC}-\eqref{eqn:superlevelsetC3}.

A candidate  RCBF  can be constructed from $h$ as follows
\begin{align}\label{eqn:forceRCBF}
	B=- \log \left(\frac{h}{1 + h} \right).
\end{align}
Because $D-\tau_{d}v_f>0$ for any $x\in\mathrm{Int}(\mathcal{C})$, it follows that
$$
L_gB(x)=\frac{\tau_{d}}{M(1 +D- \tau_{d}v_f)(D-\tau_{d}v_f)}>0,
$$
which implies that $B$ has relative degree 1 in $\mathrm{Int}(\mathcal{C})$. If  the class $\mathcal{K}$ function $\alpha_3$ in \eqref{eq:Binf} is chosen as $\gamma/B$ for some constant $\gamma>0$, then
$$
u(x) = -\frac{1}{L_g B(x)} \left(L_f B(x) - \frac{\gamma}{ B(x)} \right)
$$
provides a specific example of a $u \in \R$ satisfying
\begin{align}
	\label{eq:Binf2}
	\inf_{u \in \R} & \left[ L_f B(x) + L_g B(x) u - \frac{\gamma}{ B(x)}  \right] \leq 0.
\end{align}
As a result, $B$ is a valid RCBF for $U=\R$.

\begin{figure*}[!bt]
	\centering
	\includegraphics[width=\textwidth]{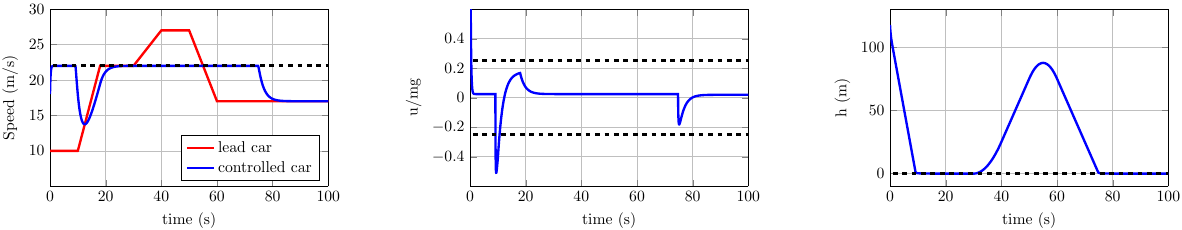}\\
	\caption{Simulation results of the ACC problem based on (ACC QP) (left) speed of the lead car and the controlled car with the desired speed $v_d$ indicated (middle)  vehicle acceleration as fractions of $g$, with typical desired  upper and lower bounds indicated  (right)  hard constraint \eqref{eqn:hard_constraint}, where positive values indicate satisfaction.}\label{fig:nobound}
\end{figure*}

\subsubsection{The CLF-CBF based QP}
\label{sec:QPACC}
As in \cite{AMPO13}, a CLF-CBF QP is constructed by combining the above constraints in the form
\begin{align}\label{eqn:accQP}
	\tag{ACC QP}
	\ubar^*(x) &= \underset{\ubar =[u,\delta]^\top \in  \R^2}{\operatorname{argmin}}   \frac{1}{2}
	\ubar^\top H_{\mathrm{acc}} \ubar + F_{\mathrm{acc}}^\top \ubar \\\nonumber
	\mathrm{s.t.} &  \quad A_{\mathrm{clf}} \ubar \leq b_{\mathrm{clf}}, \\\nonumber
	&  \quad A_{\mathrm{cbf}}  \ubar \leq b_{\mathrm{cbf}},
\end{align}
where
\begin{align}
	\label{eqn:softconstraints1}
	&A_{\mathrm{clf}} = \left[ L_gV(x) , -1 \right],\;
	b_{\mathrm{clf}} = - L_fV(x) - c V(x),
\end{align}
and
\begin{align}
	\label{eqn:hard_constraint2}
	& A_{\mathrm{cbf}} = \left[L_g B(x), 0 \right],\quad
	b_{\mathrm{cbf}}=  - L_f B(x)  + \frac{\gamma}{ B(x)}.
\end{align}

\begin{remark}  Setting $\delta= 0$ would make the CLF constraint ``hard'' in that it would force exact exponential convergence at a rate of $c$, and in such a case, if there were no inputs satisfying both the CLF constraint and the RCBF constraint, the QP would be infeasible.
\end{remark}

The cost function in the QP is selected in view of achieving the control objective encoded in the CLF, i.e., achieving the desired speed, subject to balancing the relaxation factors that ensure solvability and continuity of  \eqref{eqn:accQP}.  As explained in \cite{aaroncbfcdc14}, the CLF was constructed by first partially linearizing the system through the feedback
$u = \Fr + M \mu.$
As a result, the cost relative to this control will be chosen as $\mu^\top \mu$, which yields the following function in $u$:
$$
\mu^\top \mu = \frac{1}{M^2} \left( u^\top u - 2 u^\top \Fr + \Fr^2 \right).
$$
This can then be converted into the form
\begin{align}\label{QPcost}
	H_{\mathrm{acc}} = 2 \left[ \begin{array}{cc} \frac{1}{M^2} & 0 \\ 0 & p_{sc}  \end{array} \right] , \:\:
	F_{\mathrm{acc}} =  -2 \left[ \begin{array}{c} \frac{\Fr }{M^2}  \\ 0    \end{array} \right].
\end{align}
Here $p_{sc}$ is the weight for the relaxation $\delta$.

\newsec{Simulation Preview.}
Simulation results obtained by applying the \eqref{eqn:accQP} controller are developed in Section \ref{sec:simulation}. A sneak preview is shown in Fig. \ref{fig:nobound} to highlight a few properties of the designed controller and to motivate an important refinement. The desired cruising speed $v_d$ is set to $79.2$ (km/h), which is $22$ (m/s). The system \eqref{eqn:fgdynamics} is initialized at $(v_f(0),v_l(0),D(0))=(18,10,150)$.
The left plot in Fig \ref{fig:nobound} shows the desired cruising speed as a thick dashed line and the speeds of the lead and controlled vehicles as thin lines. The controlled vehicle achieves the desired cruising speed when it does not conflict with the time-headway constraint, and it settles to the speed of the lead car when required to maintain a safe following distance. The forward invariance of the safe set, defined by the hard constraint \eqref{eqn:hard_constraint} encoding the desired time headway $\tau_{d}$, is shown in the right plot of Fig. \ref{fig:nobound}. The middle plot shows typical ``comfort'' limits on acceleration that should be respected by the controlled vehicle, which are violated because no constraint has been imposed on the wheel force that can be requested by the QP when the car accelerates and brakes. This motivates the development of a refined barrier function that is compatible with bounds on the two vehicles' inputs.

\subsubsection{Force Constraints and CBFs}\label{subsec:modifedcbf}
%{\color{red}( Vehicle mass to M) }
The QP formulated in subsection \ref{sec:QPACC} generates a control input $u\in \R$ for the ACC-controlled vehicle. In practice, however, the wheel force that can be generated by the car is constrained by physical limits (e.g., the maximal engine torque for acceleration, maximal braking capability, and road conditions). This requires the admissible set $U$ to be bounded.  Furthermore, to guarantee driver comfort, the wheel forces the controller is allowed to apply are typically much less than the maximal forces that can be generated by the vehicle.

\newsec{Force Constraints.} We now constrain the allowable wheel forces. Supposing that we do not want to accelerate or decelerate more than some fraction of $g$, the gravitational constant, we can write the constraints on acceleration and deceleration as inequalities
\begin{align}
	\label{eqn:comfort_constraint}
	\tag{FC}
	& u  \leq a'_f M g, \quad -u  \leq a_f M g.
\end{align}
where $a_f$ and $a'_f$ are the fractions of $g$ for deceleration and acceleration, respectively. That is, the input set is now: % supposed to be
\begin{align*}
	U_{acc}:=[-a_f M g,a'_f M g].
\end{align*}

Since it may be the case that these constraints will conflict with the torque values needed to satisfy the hard constraint \eqref{eqn:hard_constraint}, we require a force-based barrier function allowing the hard constraints and force constraints to be simultaneously satisfied.  In particular, we seek a function $h_F(x)$ such that for all $x\in\C_F$, where $\C_F = \{ x~| h_F(x)\geq 0 \}$,
there exists a trajectory of \eqref{eqn:fgdynamics} satisfying \eqref{eqn:hard_constraint} and the maximum braking limit \eqref{eqn:comfort_constraint}. That is to say, within the set $\C_F$, the ACC-equipped car can \textit{always brake to maintain a desired headway using an allowed amount of deceleration}.

%{\color{blue}Two CBFs, $h_F^c$ and $h_F^o$, that can be used to define the safe set $\C_F$ are developed in \cite{aaron2016barriertac}, where $h_F^c$ and $h_F^o$ are  ``stitched'' by a set of continuously differentiable functions (refer to \cite{xu2016composition} for more details). }
In the appendix, two CBFs\footnote{The functions are piecewise defined by a set of continuously differentiable functions; more details are given in  \cite{xu2016composition}.}, $h_F^c$ and $h_F^o$,  that can be used to define the safe set $\C_F$ are developed. 
The function $h_F^c$ has a much simpler form\footnote{When the speed of the lead car is constant (i.e., $a_L=0$), and $v>v_l$, then $h_F^c$ reduces to the formula given in \cite{aaroncbfcdc14}.} than $h_F^o$, but makes a more conservative approximation of the safe set than $h_F^o$. When rolling resistance is ignored in the model \eqref{eqn:fgdynamics}, $h_F^o$ provides the maximal safe set compatible with \eqref{eqn:hard_constraint2} and the force bounds \eqref{eqn:comfort_constraint}, and will therefore be referred to as ``optimal''. The functions $h_F^o$ and $h_F^c$ in turn define the \emph{optimal RCBF} $B_F^o$ and the \emph{conservative RCBF} $B_F^c$, respectively,  using \eqref{eqn:Bmotivation} or \eqref{eqn:barrierfunc2}. The force-based hard constraints are ultimately expressed via \eqref{eqn:comfort_constraint} together with the condition
\begin{align}
	\tag{FCBF}
	\label{eqn:comfort_constraint3}
	L_f B_F(x) + L_g B_F(x) u - \frac{\gamma}{ B_F(x)} \leq 0.
\end{align}

\subsubsection{Modified CLF-CBF Based QP}\label{sec:modifiedQP}
Incorporating \eqref{eqn:comfort_constraint3} and \eqref{eqn:softconstraints1}, we have the modified force-based CLF-CBF QP:
% of the following form:
\begin{align}
	\label{eqn:accQP2}
	\tag{ACC-QP2}
	\ubar^*(x) &= \underset{\ubar = \left[ u ,\delta\right]^\top \in  U_{acc}\times\R}{\operatorname{argmin}}\; \frac{1}{2}
	\ubar^\top H_{\mathrm{acc}} \ubar + F_{\mathrm{acc}}^\top \ubar \\
	%\label{eqn:accCLF2}
	%\tag{ACC CLF}
	\nonumber
	\mathrm{s.t.} &  \quad A_{\mathrm{clf}} \ubar \leq b_{\mathrm{clf}}, \\
	%\label{eqn:accBCLF}
	%\tag{CBF}
	%&  \quad A_{\mathrm{cbf}}  \ubar \leq b_{\mathrm{cbf}} \\
	%\label{eqn:accFBCLF}
	%\tag{ACC FCBF}
	\nonumber&  \quad A_{\mathrm{fcbf}}  \ubar \leq b_{\mathrm{fcbf}}, \\
	%\label{eqn:accCC}
	%\tag{ACC FC}
	\nonumber &  \quad  A_{\mathrm{fc}}  \ubar \leq b_{\mathrm{fc}}.
\end{align}

%The remainder of this section will be devoted to constructing the constraints and cost of this QP.

%\newsec{Inequality Constraints.}
%The inequality constraints for \eqref{eqn:accQP} follow from the constraints constructed in the previous section.  In particular, following from \eqref{eqn:clfsc1}, the soft constraints result in:
%\begin{eqnarray}
%A_{\mathrm{clf}} &=& \left[ \begin{array}{cc}  \psi_{1}(x,z) & -1 \end{array} \right], \nonumber\\
%b_{\mathrm{clf}} &=& - \psi_{0}(x,z).
%\end{eqnarray}
%Following from \eqref{eqn:hard_constraint2}, the hard constraints result in:
%\begin{eqnarray}
%A_{\mathrm{cbf}} &=& \left[ \begin{array}{cc} L_g B(x,z)  & 0 \end{array} \right], \nonumber\\
%b_{\mathrm{cbf}} & = & - L_f B(x,z)  + \frac{\gamma}{ B(x,z)}.
%\end{eqnarray}

\begin{table*}[hb]
	\begin{equation}\label{eqn:LKfourstate}
		\left[  \begin{array}{c} \dot{y} \\ \dot{\nu} \\ \dot{\psi} \\ \dot{r} \end{array} \right]
		=  \left[  \begin{array}{cccc} 0 & 1 & v_0 & 0 \\
			0 & -\frac{ C_f + C_r }{ Mv_0 } & 0 & \frac{b C_r - a C_f }{ Mv_0 } - v_0 \\
			0 & 0 & 0 & 1  \\
			0 & \frac{b C_r - a C_f }{ I_{z} v_0 } & 0 & -\frac{a^2 C_f + b^2 C_r }{ I_{z} v_0 }
		\end{array} \right]
		\left[  \begin{array}{c} y \\ \nu \\ \psi \\ r \end{array} \right]
		+
		\left[  \begin{array}{c} 0 \\ \frac{C_f }{M} \\ 0 \\ a \frac{C_f }{I_{z}} \end{array} \right]  u+  \left[  \begin{array}{r} 0 \\ 0 \\ -1 \\ 0 \end{array} \right] r_{d}
	\end{equation}
\end{table*}

The soft constraints yield the same $A_{\mathrm{clf}},b_{\mathrm{clf}}$ as \eqref{eqn:softconstraints1}.
The comfort constraints in \eqref{eqn:comfort_constraint} yields $A_{\mathrm{fc}},b_{\mathrm{fc}}$ as
\begin{eqnarray}
	A_{\mathrm{fc}} = \left[ \begin{array}{cc} 1  & 0  \\  -1  & 0  \end{array} \right], \quad b_{\mathrm{fc}} = \left[ \begin{array}{c}  a'_f M g \\  a_f M g \end{array} \right].
\end{eqnarray}

By condition \eqref{eqn:comfort_constraint3}, the force constraints yield %$A_{\mathrm{fcbf}}$, $b_{\mathrm{fcbf}}$  as
%\begin{align}
%A_{\mathrm{fcbf}} &=\left[ \begin{array}{cc} L_g B_F(x)  & 0  \end{array} \right], \nonumber\\
% b_{\mathrm{fcbf}} & = - L_f B_F(x)  + \frac{1}{ B_F(x)}.\nonumber
%\end{align}
\begin{align}
	A_{\mathrm{fcbf}} &=\left[ L_g B_F(x), 0 \right], \quad
	b_{\mathrm{fcbf}} = - L_f B_F(x)  + \frac{\gamma}{ B_F(x)}.\nonumber
\end{align}

The cost function is the same as in \eqref{QPcost}. Simulation results obtained by applying the \eqref{eqn:accQP2} controller and its comparison with the \eqref{eqn:accQP} controller are provided in Section \ref{sec:simulation}.

\subsection{Lane Keeping Via QPs}
In this subsection, we consider the Lane Keeping (LK) problem using a CBF-based QP, which aims to keep the vehicle ``centered'' in a possibly curved lane. We focus on the lateral movement of the vehicle by assuming that the vehicle has a constant longitudinal speed \cite{GerdLanePotential06}.

\subsubsection{Lane Keeping Problem Setup}
%When the handling (steering or differential braking) excitation is small (e.g., does not generate more than 0.3g lateral acceleration), the roll dynamics can be ignored. In this case,
Under the assumptions of constant longitudinal speed and a linear tire-force model, a two-state handling model can be augmented to the four-state lateral-yaw model given in \eqref{eqn:LKfourstate}  \cite{GerdLanePotential06}.
	In the model, the state  is $x:=(y,\nu,\psi,r)$, where $y$ and $\psi$  are the lateral displacement and the error yaw angle in road fixed coordinates, respectively, $\nu$ is the lateral velocity, and $r$ is the yaw rate (rotation rate about the vertical axis). The input $u$ is the steering angle of the front tires, and the disturbance is the desired yaw rate $r_d$, which  can be calculated from the curvature of the road by $r_d=\frac{v_0}{R}$, where $v_0$ is a constant value for the longitudinal velocity and $R$ is the road radius of curvature. Additionally,  $M$ is the total mass of the car, $I_{z}$ is the moment of inertia about the center of mass,   $a$ and $b$ are the distance from the center of mass to the front and rear tires, respectively, and $C_r$ and $C_f$ are tire (stiffness) parameters.

The objective of the LK problem is to provide a steering input that keeps the car ``centered'' in the lane. Particularly, the car should satisfy the following hard control objective and the acceleration constraint.
%safety specification that displacement from center of lane is less than $y_{max}$, while the lateral acceleration is less than $a^{max}_y$.

\newsec{Hard Constraint.}
%This represents that the displacement of the vehicle from the lane center is always less than a constant.
This constraint requires the displacement of the vehicle from the center of the lane to be less than a given constant $y_{max}$:
%We express the constraint as
\begin{align}
	\label{eqn:hard_constraint_LK}
	\tag{LK-HC}
	|y|\leq y_{max}.
\end{align}
Since the width of US lanes is 12 feet and typical width of a car is 6 feet, we can take $y_{max}$ to be 3 feet, which is approximately 0.9 meters. Therefore, the hard constraint is $|y|\leq 0.9$.

\newsec{Acceleration Constraint.}
Due to the force limit of the car and for the comfort of the driver, the lateral acceleration needs to be bounded. We express this constraint as
\begin{align}
	\label{eqn:acce_constraint_LK}
	\tag{LK-FC}
	|\ddot{y}|\leq a_{max}.
\end{align}
%In what follows, we take $a_{max}=0.3g$.

\subsubsection{Encoding LK Constraints}
The hard and acceleration constraints can be encoded formally as follows.

\newsec{Encoding Acceleration Constraint.}
%Under the linear tire model and small angle approximation, the lateral force is equivalent to
%$C_{f}(u-\frac{\nu+ar}{v_0})-C_{r}\frac{\nu-br}{v_0}-Mv_0r$, and thus
Since
\begin{align}\label{eqn:LateralForceBalance}
	M \ddot{y}:&=C_{f}(u-\frac{\nu+ar}{v_0})-C_{r}\frac{\nu-br}{v_0}-Mv_0r_d,
\end{align}
%Therefore,
the acceleration constraint \eqref{eqn:acce_constraint_LK} is equivalent to
%$|C_{af}u-F_0|\leq Ma_{y}^{max}$ where $F_0=C_{af}\frac{\nu+ar}{v_0}+C_{ar}\frac{\nu-br}{v_0}+Mv_0r$, or equivalently,
\begin{align*}
	%\label{eqn:LKforceconstraint}
	u\in U_{lk}:&= \left[\frac{1}{C_{f}}(-Ma_{max}+F_0),\frac{1}{C_{f}}(Ma_{max}+F_0) \right]
\end{align*}
where $F_0=C_{f}\frac{\nu+ar}{v_0}+C_{r}\frac{\nu-br}{v_0}+Mv_0r_d$.

\newsec{Encoding Hard Constraint.}
%\subsection{The CBF for Lane Keeping}
%In this subsection, we first construct a RCBF satisfying the hard constraint and acceleration constraint.
Suppose that at time 0, the lateral displacement is $y(0)$ and the lateral velocity is $\dot{y}(0)$. Under the maximal allowable acceleration, it takes time $T$ for the lateral speed to be reduced to zero, where $T=\frac{|\dot{y}(0)|}{a_{max}}$. Then,
%$y(T)=y(0) +  T \dot{y}(0) + \frac{1}{2} \epsilon  T^2 a_{max}$ where $\epsilon =- {\rm sgn} \dot{y}(0)$. A straightforward calculation yields
\begin{align*}
	y(T)&=y(0) +  T \dot{y}(0) - \frac{{\rm sgn} (\dot{y}(0))}{2} T^2 a_{max}\\
	&=y(0) + \frac{1}{2}\frac{|\dot{y}(0)|}{a_{max}}\dot{y}(0).
\end{align*}
%If $\dot{y}(0)>0$, then
%$$
%y(0) + \frac{1}{2} \frac{\dot{y}^2(0)}{a_{max}} < y_{max}~~\mbox{and}~~ y(0) > -y_{max}.
%$$
%
%If $\dot{y}(0)<0$, then
%$$
%y(0) - \frac{1}{2} \frac{\dot{y}^2(0)}{a_{max}} > -y_{max}~~\mbox{and} ~~y(0) <y_{max}.
%$$
Motivated by the above formula, define
%function $h_F(y,\dot{y})$ as
\begin{align}
	h_F(x)&=\left(y_{max}-{\rm sgn}(\dot{y})~y\right)- \frac{1}{2} \frac{\dot{y}^2}{a_{max}}\label{eqn:CBFLK}
\end{align}
and $\C_F:= \{x| h_F(x)\geq 0 \}$.
%\begin{align}
%\label{eqn:SetCFforLK}
%\C_F:= \{x| h_F(x)\geq 0 \}.
%\end{align}
Then, for every $x \in \C_F$, the controlled vehicle
can remain in $\mathcal{C}_F$ while keeping the acceleration within the allowable set \eqref{eqn:acce_constraint_LK}.
%can always use allowable accelerations to remain within $\C_F$.
Indeed, differentiating  \eqref{eqn:CBFLK} for $\dot{y}\neq 0$ yields
\begin{align}
	\dot{h}_F(x,u)&=- \left( {\rm sgn}(\dot{y}) + \frac{\ddot{y}}{a_{max}} \right)\dot{y}.\label{eqn:CBFLKdot}
\end{align}
%From \eqref{eqn:LateralForceBalance} and \eqref{eqn:LKforceconstraint}, i
It follows that if $\dot{y}>0$, then $\dot{h}_F(x,u)\ge 0$ when $u=\frac{1}{C_{f}}(-Ma_{max}+F_0)$, and if $\dot{y}<0$, then $\dot{h}_F(x,u)\ge 0$ when $u=\frac{1}{C_{f}}(Ma_{max}+F_0)$. Finally, from \eqref{eqn:CBFLK}, the limit of $\dot{h}_F$ as $\dot{y}$ tends to zero is well defined and equals zero.
Taking $B_F$ as \eqref{eqn:Bmotivation} and $\gamma>0$ a constant, the above calculations imply that RCBF condition \eqref{eq:Binf} holds, namely, for any $x \in \mathrm{Int}(\C_F)$, there exists $u$ such that
\begin{align}
	\tag{LK-FCBF}
	\label{eqn:comfort_constraint_LK}
	L_f B_F(x) + L_g B_F(x) u - \frac{\gamma}{ B_F(x)} \leq 0,
\end{align}
and therefore $\mathrm{Int}(\C_F) $ is controlled invariant.

%Define $\C_{LK}:=\mathrm{Int}(\C_F) \cap\{x:|y|\leq y_{max}\}$. It is easy to prove that any feedback controller for \eqref{eqn:LKfourstate} that renders $\mathrm{Int}(\C_F)$ forward invariant also renders $\C_{LK}$ forward invariant; the proof is given in  \cite{aaron2016barriertac}.

The next proposition relates the controlled invariance of $\mathrm{Int}(\C_F) $  to the lane condition, \eqref{eqn:hard_constraint_LK}.
\begin{proposition}\label{thm:LKinvariantSet}
	Let $\C_{LK}$ be the subset of $\mathrm{Int}(\C_F) $ where the hard constraint \eqref{eqn:hard_constraint_LK} is satisfied, namely,
	\begin{align}
	\C_{LK}:=& \mathrm{Int}(\C_F)  \cap \{x:|y|\leq y_{max}\} .\label{eqn:SafeSetLK}
	\end{align}
	Then any feedback controller for \eqref{eqn:LKfourstate} that renders $\mathrm{Int}(\C_F)$ forward invariant also renders $\C_{LK}$ forward invariant.
\end{proposition}
\begin{proof}
	Figure~\ref{fig:SafeSetLK} shows the projection of $\C_F$ onto the $(y,\dot{y})$-plane, given by the strip between the upper and lower (blue) curves. The figure also shows the projection of the subset $\C_{LK}$, outlined by the dotted lines. Its forward invariance is immediate because $ \forall x\in (\partial \C_{LK}) \cap \mathrm{Int}(\C_F)$,
	${\rm sgn}(y) ~{\rm sgn}(\dot{y}) < 0.$
\end{proof}

\begin{figure}[!b]
	\centering
	\includegraphics[width=6cm]{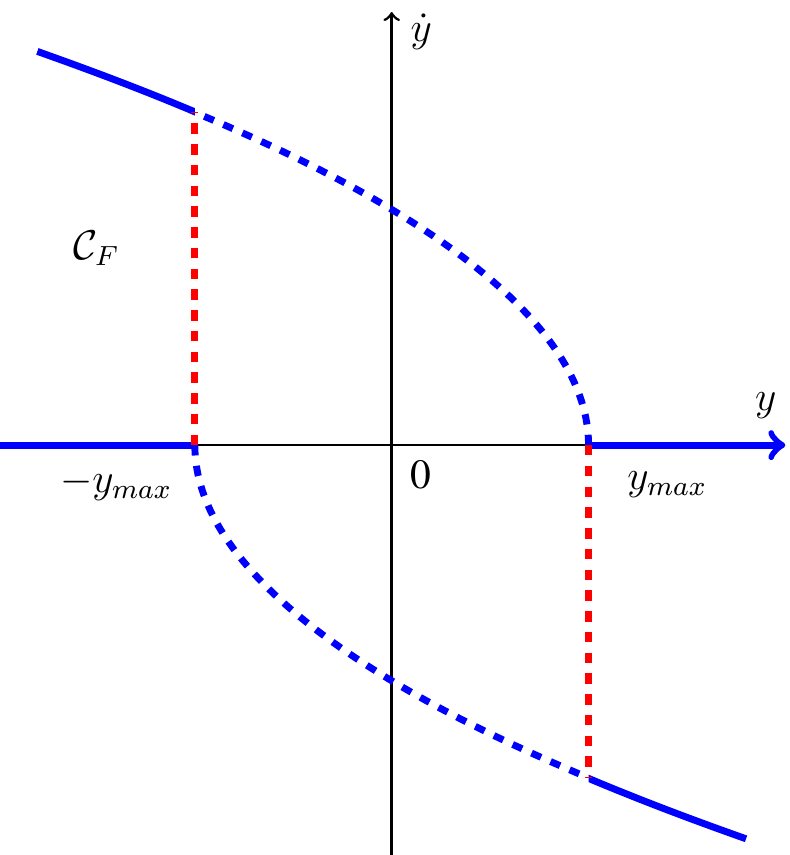}
	\caption{
		The projection of $\C_F$ onto the $(y,\dot{y})$-plane is bounded by the upper and lower curves. The subset $\C_{LK} \subset \mathrm{Int}(\C_F)$ is bounded by the dotted lines. Any feedback controller that renders $\mathrm{Int}(\C_F)$ forward invariant also renders $\C_{LK}$ forward invariant because, if $y=y_{max}$, then $\dot{y}<0$, and if $y=-y_{max}$, then $\dot{y}>0$.
	}
	\label{fig:SafeSetLK}
\end{figure}

%The next proposition relates the controlled invariance of $\mathrm{Int}(\C_F) $  to the lane condition, \eqref{eqn:hard_constraint_LK}.
%\begin{proposition}\label{thm:LKinvariantSet}
%Let $\C_{LK}$ be the subset of $\mathrm{Int}(\C_F) $ where the hard constraint \eqref{eqn:hard_constraint_LK} is satisfied, namely,
%\begin{align}
%\C_{LK}:=& \mathrm{Int}(\C_F)  \cap \{x:|y|\leq y_{max}\} .\label{eqn:SafeSetLK}
%\end{align}
%Then any feedback controller for \eqref{eqn:LKfourstate} that renders $\mathrm{Int}(\C_F)$ forward invariant also renders $\C_{LK}$ forward invariant.
%\end{proposition}
%\begin{proof}
%Figure~\ref{fig:SafeSetLK} shows the projection of $\C_F$ onto the $(y,\dot{y})$-plane, given by the strip between the upper and lower (blue) curves. The figure also shows the projection of the subset $\C_{LK}$, outlined by the dotted lines. Its forward invariance is immediate because $ \forall x\in (\partial \C_{LK}) \cap \mathrm{Int}(\C_F)$,
%${\rm sgn}(y) ~{\rm sgn}(\dot{y}) < 0.$
%\end{proof}
%
%
%\begin{figure}[!bht]
%	\centering
%	\includegraphics[width=5cm]{LK_set.pdf}
%	\caption{
%The projection of $\C_F$ onto the $(y,\dot{y})$-plane is bounded by the upper and lower curves. The subset $\C_{LK} \subset \mathrm{Int}(\C_F)$ is bounded by the dotted lines. Any feedback controller that renders $\mathrm{Int}(\C_F)$ forward invariant also renders $\C_{LK}$ forward invariant because, if $y=y_{max}$, then $\dot{y}<0$, and if $y=-y_{max}$, then $\dot{y}>0$.
%}
%\label{fig:SafeSetLK}
%\end{figure}

\begin{remark}
	Another important fact is that a feedback controller rendering $\C_{LK}$ forward invariant with bounded lateral acceleration $\ddot{y}$ results in ultimate boundedness of the yaw angle and yaw rate. Indeed, solving \eqref{eqn:LateralForceBalance} for $u$ as a function of $\ddot{y}$ and using $\dot{y}=\nu + \psi v_0$, the four-state lateral-yaw model \eqref{eqn:LKfourstate} results in
	\begin{align}
		\left[  \begin{array}{c} \dot{\psi} \\ \dot{r} \end{array} \right]
		& =  \left[  \begin{array}{cc}
			0 & 1  \\
			-\frac{(a+b)C_r}{Iz}& -\frac{b(a+b) C_r}{ I_{z} v_0 }
		\end{array} \right]
		\left[  \begin{array}{c} \psi \\ r \end{array} \right]
		\nonumber \\
		& +\left[  \begin{array}{c} 0 \\ \frac{(a+b) C_r }{ I_{z} v_0 } \end{array} \right] \dot{y}
		+\left[  \begin{array}{c} 0 \\ a\frac{M}{I_{z}} \end{array} \right] \ddot{y}
		+  \left[  \begin{array}{r}  -1 \\ 0 \end{array} \right] r_{d}.\label{eqn:LKfourstate2}
	\end{align}
	This is a linear system in companion form, driven by $\dot{y}$, $\ddot{y}$ and $r_d$.  The model parameters $a$, $b$, $C_r$, $I_z$ and $v_0$ are all positive, and hence the system is exponentially stable, and therefore input-to-state stable \cite{KHALIL01}. The term $\dot{y}$ is bounded by virtue of belonging to $\C_{LK}$ and the term $\ddot{y}$ is bounded by \eqref{eqn:acce_constraint_LK}. Because the desired yaw rate $r_d$ is bounded for bounded road curvature, it therefore follows that $\psi$ and $r$ are bounded.
\end{remark}

\newsec{Feedback Control Law For Performance.}
To illustrate that a number of  feedback controllers can be combined with a CBF via a QP, a linear controller $u=-K(x-x_{ff})$
%$$u=-K(x-x_{ff})$$
is assumed here, where $x_{ff}$ is a feedforward term,  with details given in the simulation section.  Alternatively, a quadratic Lypaunov function for the resulting closed-loop system could be used as a CLF for the open-loop system, and the feedback design completed as in Sect.~\ref{sec:ACCencodingConstraints}.

\subsubsection{CBF-based QP for LK}
Incorporating  \eqref{eqn:comfort_constraint_LK} and \eqref{eqn:acce_constraint_LK}, we have the QP for lane keeping:
% of the following form:
\begin{align}
	\label{eqn:LKQP}
	\tag{LK QP}
	\ubar^*(x) &= \underset{\ubar = \left[u , \delta\right]^\top \in  U_{lk}\times\R}{\operatorname{argmin}}   \frac{1}{2}
	\ubar^\top H_{\mathrm{lk}} \ubar + F_{\mathrm{lk}}^\top \ubar \nonumber\\
	%\label{eqn:LKFBCLF}
	%\tag{LK FCBF}
	\mathrm{s.t.}
	&\quad  A_{\mathrm{fcbf}}^{lk}  \ubar \leq b_{\mathrm{fcbf}}^{lk}, \nonumber\\
	%\label{eqn:LKCC}
	%\tag{LK FC}
	& \quad  A_{\mathrm{fc}}^{lk}  \ubar \leq b_{\mathrm{fc}}^{lk}, \nonumber\\
	%\label{ean:LKfeedback}
	%\tag{LK LF}
	&\quad u=-K(x-x_{ff})+\delta, \nonumber
\end{align}
where $\delta$ is a relaxation parameter, indicating the ``soft'' nature of this constraint, and $A_{\mathrm{fcbf}}^{lk}$, $b_{\mathrm{fcbf}}^{lk}$, $A_{\mathrm{fc}}^{lk}$, $b_{\mathrm{fc}}^{lk}$ are derived in a similar manner to the corresponding terms in Sect. \ref{sec:QPACC}.

%%
%%

%\section{Simulations For ACC and LK}
\section{Simulation Results}
\label{sec:simulation}
%\input{sections/Simulation_v02.tex}
%Simulation results obtained by applying the QP-mediated controllers for safety and performance
%are shown in this section for ACC and LK.  The parameters used for the simulation
%results are given in Table \ref{tab:para}.

Simulation results obtained by applying the QP-mediated controllers for ACC and LK
are shown in this section.  The parameters used for the simulation are given in Table \ref{tab:para}.

\subsection{Simulation results for ACC}
Various problem formulations are compared here. In all cases, the system \eqref{eqn:fgdynamics} is started from the initial condition $x(0)=(18,10,150)$.

\subsection{Comparison of two QPs}
Recall that Figure \ref{fig:nobound} showed simulation results obtained by applying the QP controller in \eqref{eqn:accQP}, where the force constraints were not taken into account. Figures \ref{fig:all_constraints} and \ref{fig:comcbf} show simulation results for \eqref{eqn:accQP2}, where the force constraints are included.

Specifically, Figure \ref{fig:all_constraints} compares \eqref{eqn:accQP} with \eqref{eqn:accQP2} using the optimal RCBF $B_F^o$ and the conservative RCBF $B_F^c$. Note that, due to limits on the wheel forces, the speed converges to $v_d$ more slowly, and begins braking earlier, as evidenced by the top plot in Fig. \ref{fig:all_constraints}.  Since RCBF $B_F^o$ is less conservative than $B_F^c$, the car maintains a smaller following distance, but the specified time-headway constraint
is always satisfied, as indicated by the bottom plot in Fig. \ref{fig:all_constraints}. Moreover, when using a force-based RCBF \eqref{eqn:forceRCBF}, the force constraints are satisfied for all time, as shown in Fig. \ref{fig:comcbf}.  Ultimately, the QP based controller \eqref{eqn:accQP2} is able to satisfy all of the control objectives and constraints for the ACC problem outlined in Sect. \ref{sec:modifiedQP} through a unified control methodology.

%Recall that Figure \ref{fig:nobound} showed simulation results obtained by applying the QP controller in \eqref{eqn:accQP} where the force constraints were not taken into account. Figures \ref{fig:all_constraints} and \ref{fig:comcbf} show simulations results for QP in \eqref{eqn:accQP2}, where the force constraints are included. Furthermore, the QP \eqref{eqn:accQP2} with force constraints is illustrated when the optimal RCBF $B_F^o$ and the conservative RCBF $B_F^c$ are used.
%
%As expected, when using a force-based RCBF \eqref{eqn:accFBCLF} and force constraints \eqref{eqn:accCC}, the force constraints are satisfied for all time (indicated by the blue trajectory of Fig. \ref{fig:comcbf}).  Note that, due to limits on the wheel forces, the speed converges to $v_d$ more slowly, and begins braking earlier (evidenced by the left plot in Fig. \ref{fig:all_constraints}).  Moreover, since RCBF $B_F^c$ is more conservative than $B_F^o$, the car maintains a smaller following distance, but the ``half-speedometer'' hard constraint
%%\textbf{[JWG to Xiangru: We need to make this more clear. I assume by hard constraint you mean the half-speedometer rule???]}
%is always satisfied (indicated by the right plot in Fig. \ref{fig:all_constraints}).  Ultimately, the QP based controller \eqref{eqn:accQP2} is able to satisfy all of the control objectives and constraints for the ACC problem outlined in Sect. \ref{sec:modifiedQP} through a unified control methodology.

%\textbf{[JWG to team: instead of non-optimal, how about ``conservative'' barrier function?]}

\begin{remark} The conference paper \cite{Aakar2015experiment} implements the above QP-based controllers in a real-time embedded processor on scaled cars. A video of the results is available on YouTube \cite{YoutubeScaledCarMovies}.
\end{remark}

%\begin{figure*}
%	\centering
%	% Requires \usepackage{graphicx}
%	\includegraphics[width=.7\textwidth]{fig7.pdf}\\
%	\caption{Comparison of QP \eqref{eqn:accQP} with QP \eqref{eqn:accQP2}.   (left) speed of the lead car and the following car (right)  hard constraint \eqref{eqn:hard_constraint} where positive values indicate satisfaction.}\label{fig:all_constraints}
%\end{figure*}

\subsection{Comparison of RCBFs and ZCBFs}\label{sec:comcbfs}
We also consider the ZCBFs for our ACC problem, which are associated with functions $h_F^o$ and $h_F^c$ given in the appendix. As expected, when using the controller from the QP \eqref{eqn:accQP2} with ZCBFs, all constraints are satisfied just as with the RCBFs. Figure \ref{fig:comcbf} shows a comparison of the generated vehicle acceleration using both types of CBFs, for both optimal and conservative cases. Our limited experience is that the ZCBFs generate a smoother input trajectory (see Fig. \ref{fig:comcbf}), while satisfying the force constraints. We suspect that this is due to the local Lipschitz constant of a RCBF potentially becoming arbitrarily large near the boundary of the safe set.

%\begin{figure}[!hbt]
%	\centering
%	% Requires \usepackage{graphicx}
%	\includegraphics[width=.4\textwidth]{figacc3.pdf}\\
%	\includegraphics[width=.4\textwidth]{figacc4.pdf}\\
%	\includegraphics[width=.42\textwidth]{figlegendacc.pdf}\\
%	\caption{Comparison of QP \eqref{eqn:accQP} with QP \eqref{eqn:accQP2}.   (top) speed of the lead car and the controlled car based on  QP \eqref{eqn:accQP} and \eqref{eqn:accQP2} (bottom)  hard constraint  \eqref{eqn:hard_constraint} based on  QP \eqref{eqn:accQP} and \eqref{eqn:accQP2} where positive values indicate satisfaction.}\label{fig:all_constraints}
%\end{figure}

%\begin{figure*}[!hbt]
%	\centering
%	% Requires \usepackage{graphicx}
%	\includegraphics[width=.24\textwidth]{figacc3.pdf}
%	\includegraphics[width=.24\textwidth]{figacc4.pdf}
%	\includegraphics[width=.24\textwidth]{figforce1.pdf}
%	\includegraphics[width=.24\textwidth]{figforce2.pdf}
%%	\includegraphics[width=.16\textwidth]{legendforce1.pdf}
%%	\includegraphics[width=.16\textwidth]{legendforce2.pdf}\\
%	\caption{Comparison of the force generated from QP \eqref{eqn:accQP2} using ZCBFs and RCBFs. (top) conservative case (bottom) optimal case}\label{fig:comcbf}
%\end{figure*}

\begin{figure}[!hbt]
	\centering
	\includegraphics[width=.35\textwidth]{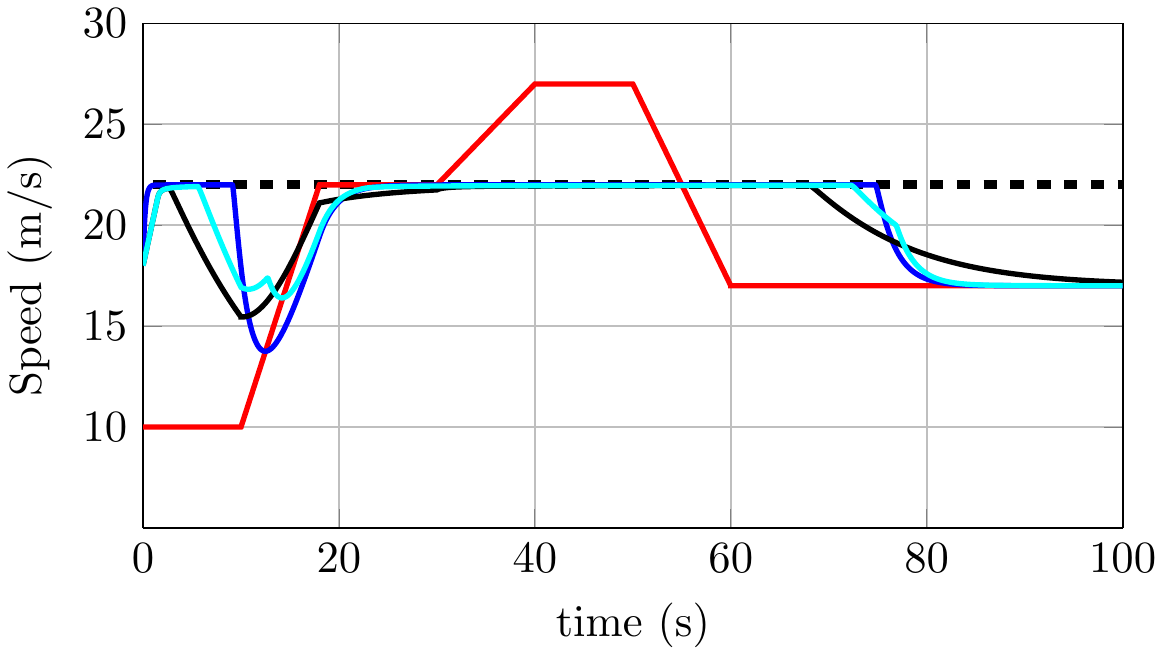}\\
	\includegraphics[width=.35\textwidth]{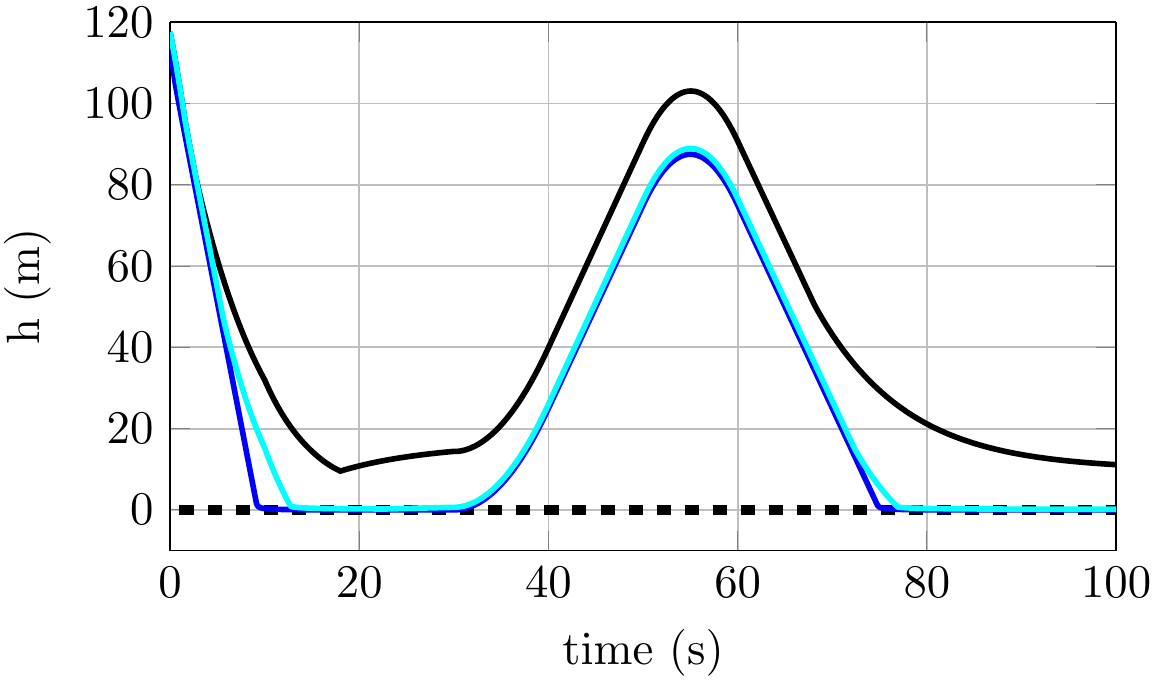}\\
	\includegraphics[width=.4\textwidth]{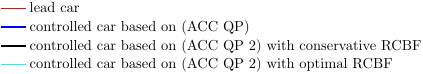}\\
	\caption{Comparison of QP \eqref{eqn:accQP} with QP \eqref{eqn:accQP2}.   (top) speed of the lead car and the controlled car based on  QP \eqref{eqn:accQP} and \eqref{eqn:accQP2} (bottom)  hard constraint  \eqref{eqn:hard_constraint} based on  QP \eqref{eqn:accQP} and \eqref{eqn:accQP2} where positive values indicate satisfaction.}\label{fig:all_constraints}
\end{figure}

\begin{figure}[!hbt]
	\centering
	\includegraphics[width=.35\textwidth]{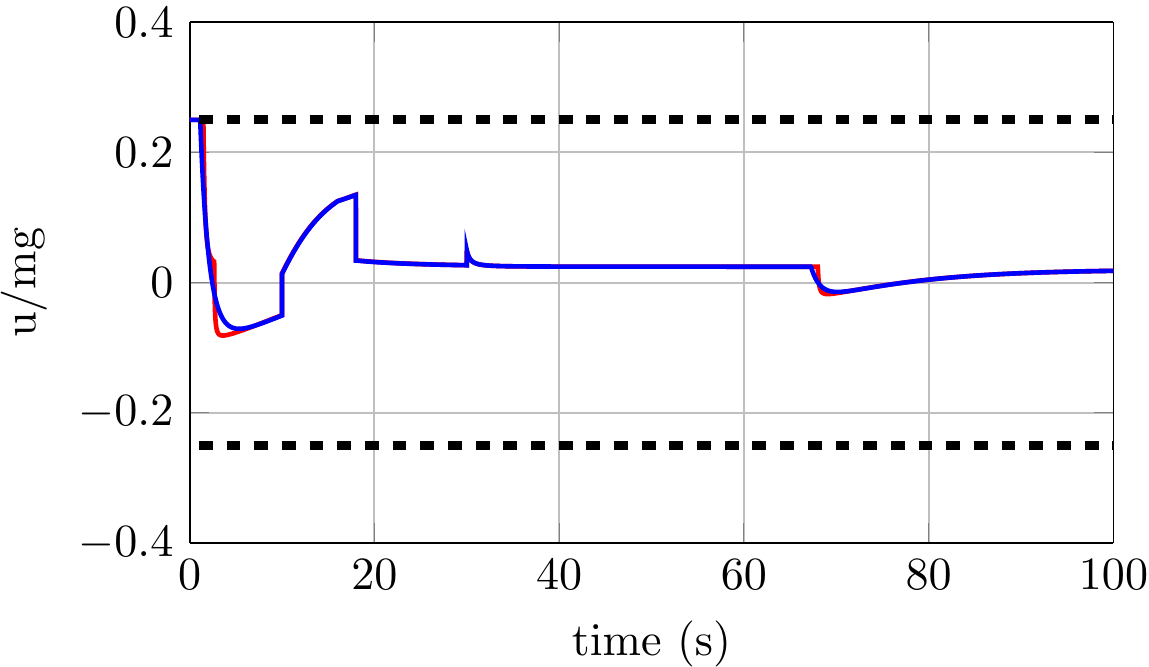}\\
	\includegraphics[width=.35\textwidth]{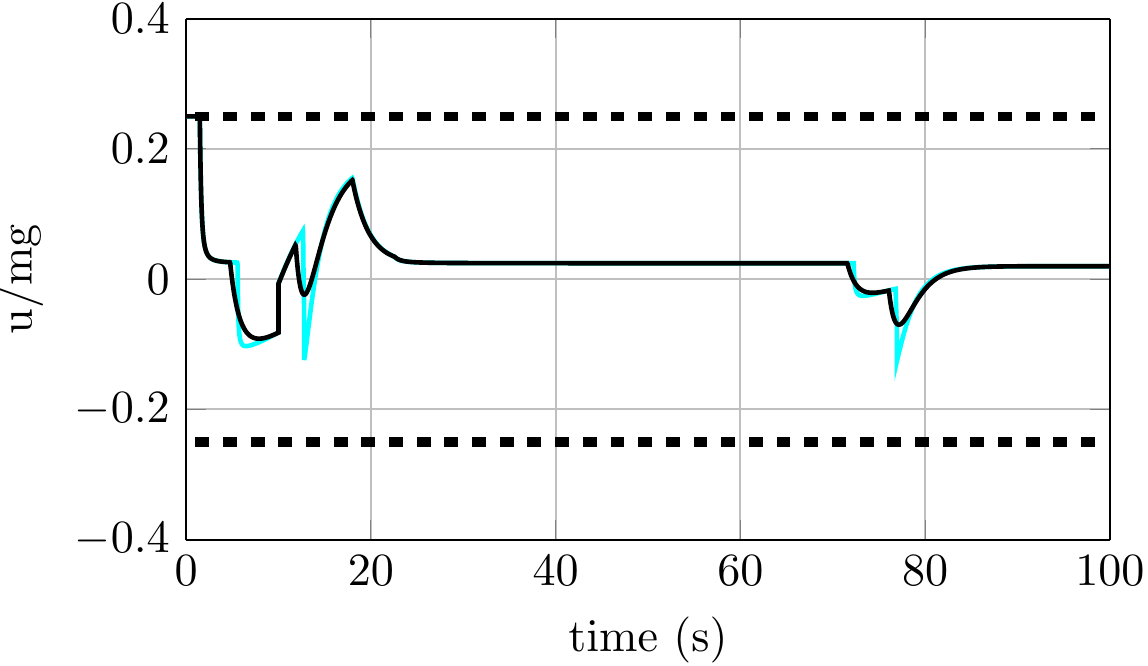}\\
	\includegraphics[width=.15\textwidth]{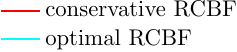}
	\includegraphics[width=.15\textwidth]{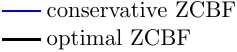}\\
	\caption{Comparison of the input force generated from QP \eqref{eqn:accQP2} using ZCBFs and RCBFs. (top) conservative CBFs (bottom) optimal CBFs}\label{fig:comcbf}
\end{figure}

\begin{figure*}[!hbt]
	\centering
	\includegraphics[width=\textwidth]{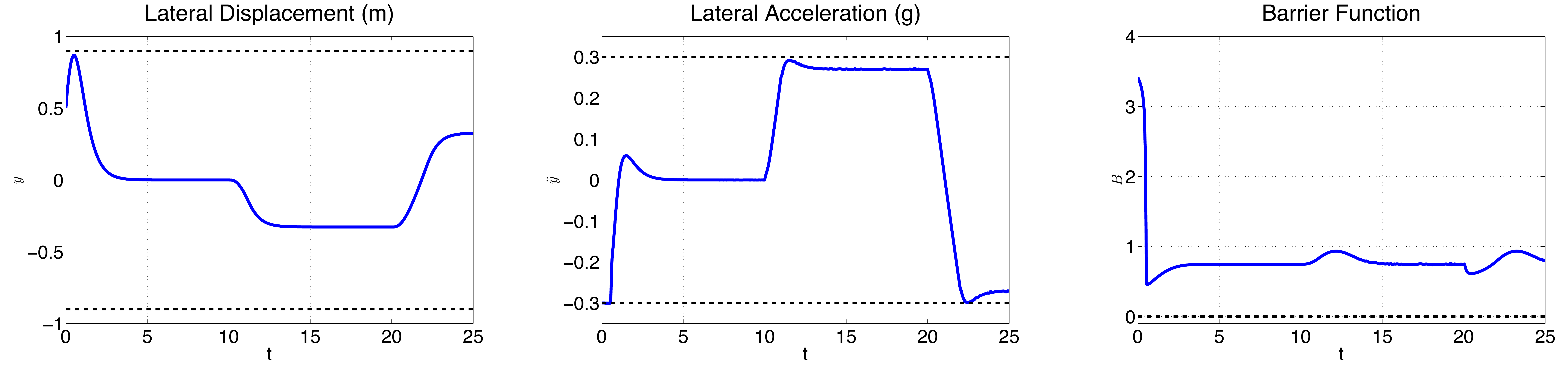}\\
	\caption{Simulation results of the QP-based controller for LK problem. (left) lateral displacement with $y^{max}=0.9 m$ (middle) lateral acceleration with $a_{max}=0.3g$ (right) barrier function where positive values indicate satisfaction.}\label{fig:LK}
\end{figure*}

\begin{table}[tb]
	\caption{Parameter values used in simulation}\label{tab:para}
	\begin{center}
		\begin{tabular}{|c|c|c|c|c|c|}
			%		\toprule
			\hline
			$M$  & 1650 $kg$ &$f_1$ & 5 $N  s/m$  &  $p_{sc}$ &  $10^2$ 	\\
			$f_0$ & 0.1 $N$ &  $f_2$ & 0.25 $N s^2/m^2$  &  $a_f'$ & 0.25 \\
			$a$ & 1.11	$m$ & $C_{f}$ & $133000$ $N/rad$ & $a_f$ & 0.25\\
			$b$ & 1.59 $m$ & $C_{r}$ & $98800$ $N/rad$ &$c$ &  $10$ \\
			$v_0$ & 27.7 $m/s$ & $I_z$ & 2315.3 $m^2\cdot kg$ &$\gamma$& $1$ \\
			$v_d$    & 22 $m/s$ & 	$g$  & 9.81 $m/s^2$  & &\\
			$y_{max}$ & 0.9 $m$ & $a_{max}$ & 0.3$\times$9.81 $m/s^2$ &&\\
			\hline
			%		\bottomrule
		\end{tabular}
	\end{center}
\end{table}

\subsection{LK simulation}
The feedback gain $K$ was determined by solving an LQR problem with control weight $R=600$ and state weight given by
$Q= K_p C^\top C + K_d C^\top A^\top A C,$
where $C=[1, 0, 20, 0], ~K_p=5, \text{and}~~K_d=0.4.$
The ``output'' $y=Cx$ corresponds to a lateral preview of approximately $0.7$ seconds. The feedforward term
$x_{ff} = [0, 0, 0, r_d]^\top$
reduces tracking error.

Simulation results for lane keeping are shown in Fig.\ref{fig:LK}. The road is assumed to be curved and the longitudinal speed of the vehicle is a constant. We can see that the absolute value of the lateral displacement is always bounded by 0.9m, and the lateral acceleration is always bounded by 0.3g. Therefore, the displacement and acceleration constraints are both satisfied.

%\begin{figure}[!hbt]
%	\centering
%	% Requires \usepackage{graphicx}
%	\includegraphics[width=.4\textwidth]{LKy.pdf}\\
%	\includegraphics[width=.4\textwidth]{LKacc.pdf}
%	\caption{Results of the QP based controller for LK problem (top) lateral displacement with $y^{max}=0.9$ (bottom) lateral acceleration with $a_y^{max}=0.3$}\label{fig:LK}
%\end{figure}

%\begin{figure*}[t!]
%	\centering
%	% Requires \usepackage{graphicx}
%	\includegraphics[width=1.0 \textwidth]{simu1.jpg}\\
%	\caption{the figure will be modifed later}
%\end{figure*}

%%
%%
\section{Conclusions}
\label{sec:conclusions}
%\input{sections/Conclusions_v02.tex}
%\textbf{[AA: I will work on this]}

This paper presented a novel framework for the control of safety-critical systems through the unification of safety conditions (expressed as control barrier functions) with control objectives (expressed as control Lyapunov functions).  At the core of this methodology was the introduction of two new classes of barrier functions: reciprocal and zeroing.  The interplay between these classes of functions was characterized, and it was shown that they provide necessary and sufficient conditions on the forward invariance of a set $\C$ under reasonable assumptions.  Therefore, in the context of (affine) control systems, this naturally yields control barrier functions (CBFs) with a large set of available control inputs that yield forward invariance of a set $\C$.  Importantly, CBFs are expressed as affine inequality constraints in the control input that--when satisfied pointwise in the candidate safe set--imply forward invariance of the set, and hence safety.
Utilizing control Lyapunov function (CLFs) to represent control objectives--which again result in affine inequality constraints in the control input--safety constraints and performance objectives were naturally unified in the framework of quadratic program (QP) based controllers. Furthermore, continuity of the resulting controllers was formally established by strictly enforcing the safety constraint and relaxing the control objective.
The mediation of safety and performance was illustrated through the application to automotive systems in the context of adaptive cruise control (ACC) and lane keeping (LK).

Future work will be devoted to building upon the foundations presented in this paper in the context of safety-critical control of cyber-physical systems, with a special focus on robotic and automotive systems.  At a formal level, this paper developed ``force-based'' barrier functions for the specific problems considered (ACC and LK), leaving as an \textit{open problem} how to do characterize and compute such functions for general classes of control systems.  In addition, formulating how to unify safety constraints, e.g., combining ACC and LK constraints into a single framework, has the potential to suggest methods for composing safety specifications.
Going beyond automotive systems, the presented methodologies are naturally applicable to robotic systems, e.g., in the context of self-collision avoidance, obstacle avoidance, end-effector (and foot) placement, and a myriad of other safety constraints.  Exploring these applications promises to provide a formal framework for safety-critical operation of robotic systems.

\section*{Acknowledgements}
The authors are indebted to H. Peng (University of Michigan) for guidance on the models and control specifications for adaptive cruise control and lane keeping. We also greatly benefited from feedback by T. Pilutti, M. Jankovic, and W. Milam (Ford Motor Company) as well as K. Butts and D. Caveney (Toyota Technical Center) during the course of this work. The work of Aaron Ames was performed while he was with the Woodruff School of Mechanical Engineering and the School of Electrical and Computer Engineering at the Georgia Institute of Technology.

\appendix

In this appendix, we develop two control barrier functions $h^c_F$ and $h^o_F$ for the safe set $\mathcal{C}$ in Section \ref{sec:ACCencodingConstraints}.
%\emph{``Control Barrier Function Based Quadratic Programs with Application to Automotive Safety Systems''}.

%Recall that the dynamics of the adaptive cruise control (ACC) model can be compactly expressed as
%%\begin{align}\label{eqn:fgdynamics}
%% \dot{x} &= \underbrace{\left( \begin{array}{c} -\Fr/M\\a_L\\x_2-x_1\end{array}\right)}_{f(x)} + \underbrace{\left( \begin{array}{c} 1/M\\0\\0\end{array}\right)}_{g(x)} u.
%%\end{align}
%\begin{align}\label{eqn:fgdynamics}
%\dot{x} &= \left( \begin{array}{c} -\Fr/M\\a_L\\x_2-x_1\end{array}\right) + \left( \begin{array}{c} 1/M\\0\\0\end{array}\right) u.
%\end{align}

Recall the dynamics of the ACC model \eqref{eqn:fgdynamics}.
If we drop the aerodynamic drag term $\Fr$, then a RCBF $B_F$ meeting the comfort  constraint \eqref{eqn:comfort_constraint}
%\begin{align}
%\label{eqn:comfort_constraint}
%\tag{FC}
%& u  \leq a'_f M g, \quad -u  \leq a_f M g.
%\end{align}
for the simplified dynamic is also a RCBF for the original model \eqref{eqn:fgdynamics}, because  the drag terms effectively augment the action of braking. In what follows, we develop the RCBFs for the ACC problem with input constraints using the simplified model, and the \emph{optimality} is in the sense of the simplified model.

%then the dynamics reduce to the simple model $\dot{v}=a_F\in[-a_fg,a'_fg]$. Because the drag terms effectively augment the action of braking, an RCBF $B_F$ meeting the comfort  constraint \eqref{eqn:comfort_constraint} for the simplified dynamic is also a RCBF for the original model \eqref{eqn:fgdynamics}.
For simplicity, suppose that the present time is $t_0=0$  and denote $v_f(0),v_l(0)$ and $D(0)$ by $v_f,v_l$, and $D$ respectively. If the lead car brakes using its maximal deceleration force, then the best response of the controlled car is to brake using its maximal deceleration force. Let $T_l$ and $T_f$ denote the time in seconds for the two cars to come to a complete stop when using maximum braking force, respectively. Supposing there are no reaction delays, it follows that
$T_l=\frac{v_l}{a_lg},T_f=\frac{v_f}{a_fg}$.

%\section{Optimal CBF for ACC}
%\vspace{.4cm}
{\newsec{Optimal CBF for ACC.}  The \emph{optimal RCBF} takes the form of $B_F^o=\frac{1}{h_F^o}$ or $B_F^o=-log(\frac{h_F^o}{1+h_F^o})$ with 
	$$
	h^o_F(x)=D-\Delta^{o*}
	$$ 
	where $\Delta^{o*}$ is given as follows:\\
	(i) If $T_f>T_l$ (or equivalently $v_f>a_lv_l/a_f$), which implies the lead car can stop first, then
	$$
	\Delta^{o*}=\max_{t\in[0,T_f]}(\Delta_1(t)+\tau_{d}(v_f-a_fgt)).
	$$
	with
	\begin{align*}
	& \Delta_1(t)=\begin{cases}
	(v_ft-\frac{1}{2}a_fgt^2)-(v_lt-\frac{1}{2}a_lgt^2), t\in[0,T_l),\\
	(v_ft-\frac{1}{2}a_fgt^2)-\frac{v_l^2}{2a_lg}, t\in[T_l,T_f].
	\end{cases}
	\end{align*}
	%Note that $\Delta_1(t)$ is the \textit{decrease in relative distance} at time $t$ between the two cars, and $\Delta^*$ is the maximum decrease in distance headway.
	
	(ii) If $T_f\leq T_l$ (or equivalently $v_f\leq a_lv_l/a_f$), which implies the following car can stop first, then
	$$
	\Delta^{o*}=\max_{t\in[0,T_f]}(\Delta_2(t)+\tau_{d}(v_f-a_fgt)),
	$$
	with
	\begin{align*}
	& \Delta_2(t)=(v_ft-\frac{1}{2}a_fgt^2)-(v_lt-\frac{1}{2}a_lgt^2),\;t\in[0,T_f].
	\end{align*}
	
	For both cases, $\Delta_1(t)$ or $\Delta_2(t)$ is the \textit{decrease in relative distance} at time $t$ between the two cars, and $\Delta^{o*}$ is the maximum decrease in distance headway.

	The explicit form of $\Delta^{o*}$ can be derived by solving the optimization problem above. The results are given as follows.
	
	Define
	\begin{align*}
	\Delta^{o*}_1&=1.8v_f,\\
	\Delta^{o*}_2&=\frac{1}{2}\frac{(1.8a_fg-v_f)^2}{a_fg}+1.8v_f-\frac{v_l^2}{2a_lg},\\
	\Delta^{o*}_3&=\frac{1}{2}\frac{(v_l+1.8a_fg-v_f)^2}{(a_f-a_l)g}+1.8v_f.
	\end{align*}
	
	%\vspace*{.3cm}
	%Next, we give the explicit form of $h_F^o$ for different cases.
	
	%\textbf{Case(I)} 
	When $a_f=a_l$,
	%%\begin{equation}
	
	%$$
	%\Delta^*=\begin{cases}
	%1.8v_f,\quad\mbox{if}\;v_f\leq v_l\;\mbox{or}\;v_l<v_f\leq v_l+1.8a_fg;\\
	%\frac{1}{2}\frac{(1.8a_fg-v_f)^2}{a_fg}+1.8v_f-\frac{v_l^2}{2a_lg},\quad\mbox{otherwise}.
	%\end{cases}
	%$$
	
	$$
	\Delta^{o*}=\begin{cases}
	\Delta^{o*}_1,\quad \mbox{if}\;0<v_f< v_l+1.8a_fg;\\
	\Delta^{o*}_2,\quad\mbox{otherwise}.
	\end{cases}
	$$
	
	%%\end{equation}

	%
	%\quad \textbf{(Ia)}$v_f\leq v_l$:
	%%$v_f\leq v_l$ implies that $v_f-v_l-1.8a_fg<0$. $\Delta^*= \max_{t\in[0,\frac{v_f}{a_fg}]} S_1(t)=\max_{t\in[0,\frac{v_f}{a_fg}]}(v_f-v_l-1.8a_fg)t+1.8v_f$. Then $\max_{t\in[0,\frac{v_f}{a_fg}]} S_1(t)=S_1(0)$. That is,
	%$$
	%\Delta^*=1.8v_f
	%$$
	%
	%\quad \textbf{(Ib)}$v_f>v_l$\\
	%
	%\quad \quad \textbf{(Ib-1)}$v_l<v_f\leq v_l+1.8a_fg$:
	%%$v_l<v_f\leq v_l+1.8a_fg$ implies that $v_f-v_l-1.8a_fg\leq 0$. Then $\max_{t\in[0,\frac{v_l}{a_lg}]} S_1(t)=S_1(0)=1.8v_f$, $\max_{t\in[\frac{v_l}{a_lg},\frac{v_f}{a_fg}]} S_2(t)=S_2(\frac{v_l}{a_lg})\leq S_1(0)$. That is,
	%$$
	%\Delta^*=1.8v_f
	%$$
	%
	%\quad \quad \textbf{(Ib-2)}$v_f> v_l+1.8a_f$:
	%%$v_f> v_l+1.8a_fg$ implies that $v_f-v_l-1.8a_fg>0$. Then $\Delta^*=\max_{t\in[\frac{v_l}{a_lg},\frac{v_f}{a_fg}]} S_2(t)=S_2(\frac{v_f-1.8a_fg}{a_fg})$. That is,
	%$$
	%\Delta^*=\frac{1}{2}\frac{(1.8a_fg-v_f)^2}{a_fg}+1.8v_f-\frac{v_l^2}{2a_lg}
	%$$
	
	%\textbf{Case(II)} 
	When $a_f>a_l$,
	$$
	\Delta^{o*}=\begin{cases}
	\Delta^{o*}_1,\quad \mbox{if}\;0<v_f< v_l+1.8a_fg;\\
	\Delta^{o*}_2,\quad \mbox{if}\;v_f\geq\frac{a_f}{a_l}v_l+1.8a_fg;\\
	\Delta^{o*}_3,\quad \mbox{otherwise}.
	\end{cases}
	$$
	
	%\quad \textbf{(IIa)} $v_f<\frac{a_f}{a_l}v_l$\\
	%
	%\quad \quad \textbf{(IIa-1)} $v_l\geq v_f$:
	%%$\Delta^*=\max_{t\in[0,\frac{v_f}{a_fg}]} S_1(t)=S_1(0)=1.8v_f$. That is,
	%$$
	%\Delta^*=1.8v_f
	%$$
	%
	%\quad \quad \textbf{(IIa-2)} $v_l< v_f$\\
	%
	%\quad \quad \quad \quad \textbf{(IIa-2-1)} $v_f\leq v_l+1.8a_fg$:
	%%$\Delta^*=\max_{t\in[0,\frac{v_f}{a_fg}]} S_1(t)=S_1(0)=1.8v_f$. That is,
	%$$
	%\Delta^*=1.8v_f
	%$$
	%
	%\quad \quad \quad \quad \textbf{(IIa-2-2)} $v_f>v_l+1.8a_fg$:
	%%Let $v_f-1.8a_fg-a_fgt^*=v_l-a_lgt^*$, then $t^*=\frac{v_f-1.8a_fg-v_l}{(a_f-a_l)g}$. $\Delta^*=\max_{t\in[0,\frac{v_f}{a_fg}]} S_1(t)=S_1(t^*)$. That is,
	%$$
	%\Delta^*=\frac{1}{2}\frac{(v_l+1.8a_fg-v_f)^2}{(a_f-a_l)g}+1.8v_f
	%$$
	%
	%\quad \textbf{(IIb)} $v_f\geq\frac{a_f}{a_l}v_l$\\
	%
	%\quad \quad \textbf{(IIb-1)} $v_f\leq v_l+1.8a_fg$: 
	%$$
	%\Delta^*=1.8v_f
	%$$
	%
	%\quad \quad \textbf{(IIb-2)} $v_f>v_l+1.8a_fg$ \\
	%
	%\quad \quad \quad \quad \textbf{(IIb-2-1)} $\frac{v_f-1.8a_fg}{a_fg}\leq \frac{v_l}{a_lg}$: 
	%%$\Delta^*=S_1(\frac{v_f-1.8a_fg-v_l}{(a_f-a_l)g})$. That is,
	%$$
	%\Delta^*=\frac{1}{2}\frac{(v_l+1.8a_fg-v_f)^2}{(a_f-a_l)g}+1.8v_f
	%$$
	%
	%\quad \quad \quad \quad \textbf{(IIb-2-2)} $\frac{v_f-1.8a_fg}{a_fg}> \frac{v_l}{a_lg}$: 
	%%$\Delta^*=\max_{t\in[\frac{v_l}{a_lg},\frac{v_f}{a_fg}]} S_2(t)=S_2(\frac{v_f-1.8a_fg}{a_fg})$. That is,
	%$$
	%\Delta^*=\frac{1}{2}\frac{(1.8a_fg-v_f)^2}{a_fg}+1.8v_f-\frac{v_l^2}{2a_lg}
	%$$
	
	%$\Delta^*=S_2(\frac{v_f-1.8a_fg}{a_fg})$. That is,
	%$$
	%\Delta^*=\frac{1}{2}\frac{(1.8a_fg-v_f)^2}{a_f^2}+1.8v_f-\frac{v_l^2}{2a_l}
	%$$
	
	%\textbf{Case(III)} $a_f<a_l$\\
	When $a_f<a_l$,
	$$
	\Delta^{o*}=\begin{cases}
	\Delta^{o*}_1,\quad \mbox{if}\;0<v_f< \sqrt{a_f/a_l}v_l +1.8a_fg ;\\
	\Delta^{o*}_2,\quad \mbox{otherwise}.
	\end{cases}
	$$

	{\newsec{Convervative  CBF for ACC.} 
		A simpler, albeit conservative, estimate of the safe set is given in this section.
		Replace the term $\tau_{d}(v_f-a_fgt)$ by $\tau_{d}v_f$ in $\Delta^{o*}$ and denote the new formula by $\Delta^{c*}$. Then we define 
		$$
		h^c_F(x)=D-\Delta^{c*},
		$$  
		and call the associated RCBF candidate $B_F^c$ as in $B_F^c:=1/h_F^c$  or $B_F^c=\log(\frac{h_F^c}{1+h_F^c})$ the \emph{conservative RCBF}. Clearly, $\Delta^{c*}$ is larger than $\Delta^{o*}$ and therefore $h_F^c$  corresponds to a more conservative safety set than $h_F^o$. 
		However, the closed-form expressions for $\Delta^{c*}$ are much simpler than $\Delta^{o*}$, which is given as follows.
		
		$$
		\Delta^{c*}=\begin{cases}
		\tau_{d}v_f,\quad \mbox{if}\;v_l\geq v_f,T_l\geq T_f;\\
		\tau_{d}v_f + \frac{1}{2}\frac{(a_lv_f-a_fv_l)^2}{a_la_f(a_l-a_f)g},\quad \mbox{if}\;v_l\geq v_f,T_l< T_f;\\
		\tau_{d}v_f + \frac{1}{2}\frac{(v_f-v_l)^2}{(a_f-a_l)g},\quad \mbox{if}\;v_l<v_f,T_l\geq T_f;\\
		\tau_{d}v_f + \frac{1}{2}\frac{v_f^2a_l-v_l^2a_f}{a_fa_lg},\quad \mbox{if}\;v_l<v_f,T_l< T_f.
		\end{cases}
		$$
		
		%The closed-form expressions for $h^c_F$ are much simpler than those for $h_F^o$ and are given below:
		%
		%Case (i). $v_l\geq v_f,T_l\geq T_f$:
		%\begin{equation*}
		%h_F^c(x)=D-\tau_{d}v_f.
		%\end{equation*}
		%
		%Case (ii). $v_l\geq v_f,T_l< T_f$:
		%\begin{equation*}
		%h_F^c(x)=D-\tau_{d}v_f-\frac{1}{2}\frac{(a_lv_f-a_fv_l)^2}{a_la_f(a_l-a_f)g}.
		%\end{equation*}
		%
		%
		%Case (iii). $v_l<v_f,T_l\geq T_f$:
		%\begin{equation*}
		%h_F^c(x)=D-\tau_{d}v_f-\frac{1}{2}\frac{(v_f-v_l)^2}{(a_f-a_l)g}.
		%\end{equation*}
		%
		%Case (iv). $v_l<v_f,T_l<T_f$:
		%\begin{equation*}
		%h_F^c(x)=D-\tau_{d}v_f-\frac{1}{2}\frac{v_f^2a_l-v_l^2a_f}{a_fa_lg}.
		%\end{equation*}
		%
		%Note that in Case (ii),  $v_l\geq v_f$ and $T_l< T_f$ imply that $a_f \neq a_l$. Similar reasoning applies in Case (iii).
		
		%\vspace*{0.5cm}
		%\section{Proof of Validity of CBFs }
		%{\newsec{Proof of Validity of CBFs.} 
		It is easy to check that $h_F^o$ and $h_F^c$ above are valid RCBFs. For example, we prove that $B_F^c:=1/h_F^c$ satisfies
		\begin{align*}
		\inf_{u \in U}  \left[ L_f B_F^c + L_g B_F^c u - \frac{1}{ B_F^c}  \right] \leq 0.
		\end{align*}
		%where $h_F^c$ is given above and $u$ satisfying (FC):
		
		%\vspace*{0.5cm}
		When $v_l\geq v_f,T_l\geq T_f$, taking $u=-a_fMg$ results in
		$$
		L_f B_F^c + L_g B_F^c u=-\frac{\tau_{d}a_fg+(v_l-v_f)}{(h_F^c)^2}<0.
		$$
		
		When $v_l\geq v_f,T_l< T_f$, taking $u=-a_fMg$ results in
		\begin{align*}
		&L_f B_F^c + L_g B_F^c u\\
		=&-\frac{\tau_{d}a_fg+(v_l-v_f)+\frac{(a_lv_f-a_fv_l)(a_la_fg+a_fa_L)}{a_la_f(a_l-a_f)g}}{(h_F^c)^2}<0,
		\end{align*}
		because $a_l>a_f$ and $a_lv_f>a_fv_l$ in this case.
		
		When $v_l<v_f,T_l\geq T_f$, taking $u=-a_fMg$ results in
		$$
		L_f B_F^c + L_g B_F^c u=-\frac{\tau_{d}a_fg+\frac{(v_f-v_l)(a_lg+a_L)}{(a_f-a_l)g}}{(h_F^c)^2}<0,
		$$
		because $v_f>v_l$ and $a_f>a_l$ in this case.
		
		When $v_l<v_f,T_l<T_f$, taking $u=-a_fMg$ results in
		$$
		L_f B_F^c + L_g B_F^c u=-\frac{\tau_{d}a_fg+v_l+v_la_L/a_lg}{(h_F^c)^2}<0.
		$$

Clearly, the corresponding optimal and conservative ZCBFs are just $h_F^o$ and $h_F^c$ shown above, respectively. 

%\section*{Appendix}
%\appendices
%\label{appex:CBF}
%\input{sections/Appen_short.tex}

\bibliographystyle{IEEEtran}
%\bibliography{./AXGT_TAC_CBF_ref}AXGT_TAC_CBF_ref
\bibliography{./AXGT_TAC_CBF_ref}

\begin{thebibliography}{10}
\providecommand{\url}[1]{#1}
\csname url@samestyle\endcsname
\providecommand{\newblock}{\relax}
\providecommand{\bibinfo}[2]{#2}
\providecommand{\BIBentrySTDinterwordspacing}{\spaceskip=0pt\relax}
\providecommand{\BIBentryALTinterwordstretchfactor}{4}
\providecommand{\BIBentryALTinterwordspacing}{\spaceskip=\fontdimen2\font plus
\BIBentryALTinterwordstretchfactor\fontdimen3\font minus
  \fontdimen4\font\relax}
\providecommand{\BIBforeignlanguage}[2]{{%
\expandafter\ifx\csname l@#1\endcsname\relax
\typeout{** WARNING: IEEEtran.bst: No hyphenation pattern has been}%
\typeout{** loaded for the language `#1'. Using the pattern for}%
\typeout{** the default language instead.}%
\else
\language=\csname l@#1\endcsname
\fi
#2}}
\providecommand{\BIBdecl}{\relax}
\BIBdecl

\bibitem{Romd2014CDCunitiing}
M.~Z. Romdlony and B.~Jayawardhana, ``Uniting control {L}yapunov and control
  barrier functions,'' in \emph{IEEE Conference on Decision and Control}, 2014,
  pp. 2293--2298.

\bibitem{aaroncbfcdc14}
A.~D. Ames, J.~W. Grizzle, and P.~Tabuada, ``Control barrier function based
  quadratic programs with application to adaptive cruise control,'' in
  \emph{IEEE Conference on Decision and Control}, 2014, pp. 6271--6278.

\bibitem{FGW02}
A.~Forsgren, P.~E. Gill, and M.~H. Wright, ``Interior methods for nonlinear
  optimization,'' \emph{SIAM Review}, vol.~44, no.~4, pp. 525--597, 2002.

\bibitem{MCP16}
P.~Malisani, F.~Chaplais, and N.~Petit, ``An interior penalty method for
  optimal control problems with state and input constraints of nonlinear
  systems,'' \emph{Optimal Control Applications and Methods}, vol.~37, pp.
  3--33, 2016.

\bibitem{Tee}
K.~P. Tee, S.~S. Ge, and E.~H. Tay, ``Barrier {L}yapunov functions for the
  control of output-constrained nonlinear systems,'' \emph{Automatica},
  vol.~45, no.~4, pp. 918 -- 927, 2009.

\bibitem{wieland2007constructive}
P.~Wieland and F.~Allg{\"o}wer, ``Constructive safety using control barrier
  functions,'' in \emph{Proceedings of the 7th IFAC Symposium on Nonlinear
  Control System}, 2007, pp. 462--467.

\bibitem{aubin2009viability}
J.~P. Aubin, \emph{Viability theory}.\hskip 1em plus 0.5em minus 0.4em\relax
  Springer, 2009.

\bibitem{Sloth2012Composite}
C.~Sloth, R.~Wisniewski, and G.~J. Pappas, ``On the existence of compositional
  barrier certificates,'' in \emph{IEEE Conference on Decision and Control},
  2012, pp. 4580--4585.

\bibitem{wisniewskiconverse}
R.~Wisniewski and C.~Sloth, ``Converse barrier certificate theorems,''
  \emph{IEEE Transactions on Automatic Control}, vol.~61, no.~5, pp.
  1356--1361, 2016.

\bibitem{Prajna2007siam}
S.~Prajna and A.~Rantzer, ``Convex programs for temporal verification of
  nonlinear dynamical systems,'' \emph{SIAM Journal of Control and
  Optimization}, vol.~46, no.~3, pp. 999--1021, 2007.

\bibitem{prajna2007framework}
S.~Prajna, A.~Jadbabaie, and G.~J. Pappas, ``A framework for worst-case and
  stochastic safety verification using barrier certificates,'' \emph{Automatic
  Control, IEEE Transactions on}, vol.~52, no.~8, pp. 1415--1428, 2007.

\bibitem{PSV:CDC:2013}
D.~Panagou, D.~M. Stipanovi{\v{c}}, and P.~G. Voulgaris, ``Multi-objective
  control for multi-agent systems using lyapunov-like barrier functions,'' in
  \emph{IEEE Conference on Decision and Control}, 2013, pp. 1478--1483.

\bibitem{Sontag:firstCLF}
E.~Sontag, ``A {L}yapunov-like stabilization of asymptotic controllability,''
  \emph{SIAM Journal of Control and Optimization}, vol.~21, no.~3, pp.
  462--471, 1983.

\bibitem{artstein1983stabilization}
Z.~Artstein, ``Stabilization with relaxed controls,'' \emph{Nonlinear Analysis:
  Theory, Methods \& Applications}, vol.~7, no.~11, pp. 1163--1173, 1983.

\bibitem{Sontag:universal}
E.~Sontag, ``A 'universal' contruction of {A}rtstein's theorem on nonlinear
  stabilization,'' \emph{Systems \& Control Letters}, vol.~13, pp. 117--123,
  1989.

\bibitem{FK:Book}
R.~A. Freeman and P.~V. Kokotovi\'{c}, \emph{Robust Nonlinear Control
  Design}.\hskip 1em plus 0.5em minus 0.4em\relax Birkh\"{a}user, 1996.

\bibitem{BarrierRevisited}
L.~Dai, T.~Gan, B.~Xia, and N.~Zhan, ``Barrier certificates revisited,''
  \emph{Journal of Symbolic Computation}, 2016,
  \url{http://dx.doi.org/10.1016/j.jsc.2016.07.010}.

\bibitem{KongExpoBarrier13}
H.~Kong, F.~He, X.~Song, W.~Hung, and M.~Gu, ``Exponential-condition-based
  barrier certificate generation for safety verification of hybrid systems,''
  in \emph{CAV2013, volume 8044 of LNCS}.\hskip 1em plus 0.5em minus
  0.4em\relax Springer, 2013, pp. 242--257.

\bibitem{AmGaGr:CDC12}
A.~D. Ames, K.~Galloway, and J.~W. Grizzle, ``{C}ontrol {L}yapunov {F}unctions
  and {H}ybrid {Z}ero {D}ynamics,'' in \emph{IEEE Conference on Decision and
  Control}, 2012, pp. 6837--6842.

\bibitem{AmGaGrSr:TAC12}
A.~D. Ames, K.~Galloway, J.~W. Grizzle, and K.~Sreenath, ``Rapidly
  {E}xponentially {S}tabilizing {C}ontrol {L}yapunov {F}unctions and {H}ybrid
  {Z}ero {D}ynamics,'' \emph{IEEE Trans. Automatic Control}, vol.~59, no.~4,
  pp. 115--1130, 2014.

\bibitem{YouTubeExp1}
K.~Galloway, K.~Sreenath, A.~Ames, and J.~Grizzle, ``{W}alking with {C}ontrol
  {L}yapunov {F}unctions (2012),'' {Youtube Video. [Online].}, available:
  \url{http://youtu.be/onOd7xWbGAk}.

\bibitem{Kevin2015torque}
K.~Galloway, K.~Sreenath, A.~D. Ames, and J.~W. Grizzle, ``Torque saturation in
  bipedal robotic walking through control {L}yapunov function-based quadratic
  programs,'' \emph{IEEE Access}, vol.~3, pp. 323--332, 2015.

\bibitem{AMPO13}
A.~D. Ames and M.~Powell, ``Towards the unification of locomotion and
  manipulation through control lyapunov functions and quadratic programs,'' in
  \emph{Control of Cyber-Physical Systems, Lecture Notes in Control and
  Information Sciences}, vol. 449, 2013, pp. 219--240.

\bibitem{morris2015continuity}
B.~J. Morris, M.~J. Powell, and A.~D. Ames, ``Continuity and smoothness
  properties of nonlinear optimization-based feedback controllers,'' in
  \emph{IEEE Conference on Decision and Control}, 2015, pp. 151--158.

\bibitem{mareczek2002invariance}
J.~Mareczek, M.~Buss, and M.~W. Spong, ``Invariance control for a class of
  cascade nonlinear systems,'' \emph{Automatic Control, IEEE Transactions on},
  vol.~47, no.~4, pp. 636--640, 2002.

\bibitem{wolff2005invariance}
J.~Wolff and M.~Buss, ``Invariance control design for constrained nonlinear
  systems,'' in \emph{Proceedings of the 16th IFAC World Congress}, 2005, pp.
  37--42.

\bibitem{kimmel2014invariance}
M.~Kimmel and S.~Hirche, ``Invariance control with chattering reduction,'' in
  \emph{IEEE Decision and Control}, 2014, pp. 68--74.

\bibitem{lee2012rollover}
S.~Lee, M.~Leibold, M.~Buss, and F.~C. Park, ``Rollover prevention of mobile
  manipulators using invariance control and recursive analytic {ZMP}
  gradients,'' \emph{Advanced Robotics}, vol.~26, no. 11-12, pp. 1317--1341,
  2012.

\bibitem{kimmel2015active}
M.~Kimmel and S.~Hirche, ``Active safety control for dynamic human-robot
  interaction,'' in \emph{Intelligent Robots and Systems (IROS), IEEE/RSJ
  International Conference on}, 2015, pp. 4685--4691.

\bibitem{IsidoriNonControl95}
A.~Isidori, \emph{Nonlinear Control Systems-3rd Edition}.\hskip 1em plus 0.5em
  minus 0.4em\relax Springer, 1995.

\bibitem{vahidi2003research}
A.~Vahidi and A.~Eskandarian, ``Research advances in intelligent collision
  avoidance and adaptive cruise control,'' \emph{Intelligent Transportation
  Systems, IEEE Transactions on}, vol.~4, no.~3, pp. 143--153, 2003.

\bibitem{li2011model}
S.~Li, K.~Li, R.~Rajamani, and J.~Wang, ``Model predictive multi-objective
  vehicular adaptive cruise control,'' \emph{Control Systems Technology, IEEE
  Transactions on}, vol.~19, no.~3, pp. 556--566, 2011.

\bibitem{naus2010design}
G.~Naus, J.~Ploeg, M.~Van~de Molengraft, W.~Heemels, and M.~Steinbuch, ``Design
  and implementation of parameterized adaptive cruise control: An explicit
  model predictive control approach,'' \emph{Control Engineering Practice},
  vol.~18, no.~8, pp. 882--892, 2010.

\bibitem{ioannou1993autonomous}
P.~A. Ioannou and C.-C.~C. Chien, ``Autonomous intelligent cruise control,''
  \emph{Vehicular Technology, IEEE Transactions on}, vol.~42, no.~4, pp.
  657--672, 1993.

\bibitem{liang1999optimal}
C.-Y. Liang and H.~Peng, ``Optimal adaptive cruise control with guaranteed
  string stability,'' \emph{Vehicle System Dynamics}, vol.~32, no. 4-5, pp.
  313--330, 1999.

\bibitem{liang2000string}
------, ``String stability analysis of adaptive cruise controlled vehicles,''
  \emph{JSME International Journal Series C}, vol.~43, no.~3, pp. 671--677,
  2000.

\bibitem{van2006impact}
B.~Van~Arem, C.~J. van Driel, and R.~Visser, ``The impact of cooperative
  adaptive cruise control on traffic-flow characteristics,'' \emph{Intelligent
  Transportation Systems, IEEE Transactions on}, vol.~7, no.~4, pp. 429--436,
  2006.

\bibitem{GerdLanePotential06}
E.~J. Rossetter and C.~J. Gerdes, ``Lyapunov based performance guarantees for
  the potential field lane-keeping assistance system,'' \emph{Journal of
  Dynamic Systems, Measurement, and Control}, vol. 128, no.~3, pp. 510--522,
  2006.

\bibitem{Xu2015ADHS}
X.~Xu, P.~Tabuada, A.~D. Ames, and J.~W. Grizzle, ``Robustness of control
  barrier functions for safety critical control,'' in \emph{IFAC Conference on
  Analysis and Design of Hybrid Systems}, 2015, pp. 54--61.

\bibitem{boyd2004convex}
S.~P. Boyd and L.~Vandenberghe, \emph{Convex optimization}.\hskip 1em plus
  0.5em minus 0.4em\relax Cambridge university press, 2004.

\bibitem{KHALIL01}
H.~K. Khalil, \emph{Nonlinear Systems (Third Edition)}.\hskip 1em plus 0.5em
  minus 0.4em\relax Prentice Hall, 2002.

\bibitem{agarwal1993uniqueness}
R.~P. Agarwal, R.~P. Agarwal, and V.~Lakshmikantham, \emph{Uniqueness and
  nonuniqueness criteria for ordinary differential equations}.\hskip 1em plus
  0.5em minus 0.4em\relax World Scientific, 1993.

\bibitem{hartman2002ODE}
P.~Hartman, \emph{Ordinary Differential Equations (Second Edition)}.\hskip 1em
  plus 0.5em minus 0.4em\relax SIAM, 2002.

\bibitem{BlanchiniBook08}
F.~Blanchini and M.~Stefano, \emph{Set Theoretic Methods In Control}.\hskip 1em
  plus 0.5em minus 0.4em\relax Birkh\"{a}user Basel, 2008.

\bibitem{BlanchiniAuto99}
F.~Blanchini, ``Set invariance in control,'' \emph{Automatica}, vol.~35,
  no.~11, pp. 1747--1767, 1999.

\bibitem{SmoothConverse96}
Y.~Lin, E.~Sontag, and Y.~Wang, ``A smooth converse {L}yapunov theorem for
  robust stability,'' \emph{SIAM Journal on Control and Optimization}, vol.~34,
  no.~1, pp. 124--160, 1996.

\bibitem{Bacc2005Lyapunovfunctionbook}
A.~Bacciotti and L.~Rosier, \emph{Liapunov functions and stability in control
  theory}.\hskip 1em plus 0.5em minus 0.4em\relax Springer, 2005.

\bibitem{Isidori1999NonlinearBook2}
A.~Isidori, \emph{Nonlinear control systems II}.\hskip 1em plus 0.5em minus
  0.4em\relax Springer, 1999.

\bibitem{hsubackstepping}
S.-C. Hsu, X.~Xu, and A.~D. Ames, ``Control barrier functions based quadratic
  programs with application to bipedal robotic walking,'' in \emph{American
  Control Conference}, 2015, pp. 4542--4548.

\bibitem{nguyen2016exponential}
Q.~Nguyen and K.~Sreenath, ``Exponential control barrier functions for
  enforcing high relative-degree safety-critical constraints,'' in
  \emph{American Control Conference}, 2016, pp. 322--328.

\bibitem{FreemanSIAM96}
R.~A. Freeman and P.~V. Kokotovic, ``Inverse optimality in robust
  stabilization,'' \emph{SIAM Journal on Control and Optimization}, vol.~34,
  no.~4, pp. 1365--1391, 1996.

\bibitem{Spong86}
M.~W. Spong, J.~S. Thorp, and J.~M. Kleinwaks, ``The control of robot
  manipulators with bounded input,'' \emph{IEEE Trans. Automatic Control},
  vol.~31, no.~6, pp. 483--490, 1986.

\bibitem{Aakar2015experiment}
A.~Mehra, W.-L. Ma, F.~Berg, P.~Tabuada, J.~W. Grizzle, and A.~D. Ames,
  ``Adaptive cruise control: Experimental validation of advanced controllers on
  scale-model cars,'' in \emph{American Control Conference}, 2015, pp.
  1411--1418.

\bibitem{LuenbergerOptimizationVectorSpaceMethods}
D.~G. Luenberger, \emph{Optimization by vector space methods}.\hskip 1em plus
  0.5em minus 0.4em\relax John Wiley \& Sons, 1969.

\bibitem{Vogel}
K.~Vogel, ``A comparison of headway and time to collision as safety
  indicators,'' \emph{Accident Analysis \& Prevention}, vol.~35, no.~3, pp. 427
  -- 433, 2003.

\bibitem{xu2016composition}
\BIBentryALTinterwordspacing
X.~Xu, J.~W. Grizzle, P.~Tabuada, and A.~D. Ames, ``Correctness guarantees for
  the composition of lane keeping and adaptive cruise contro,'' 2016. [Online].
  Available: \url{http://arxiv.org/abs/1609.06807}
\BIBentrySTDinterwordspacing

\bibitem{YoutubeScaledCarMovies}
``Adaptive cruise control: Experimental validation of advanced con- trollers on
  scale-model cars.'' \url{http://youtu.be/9Du7F76s4jQ}.

\end{thebibliography}

\end{document}